\documentclass[fleqn]{llncs}
\usepackage{imdm}
\usepackage{version}

\title{Inversive Meadows and Divisive Meadows}
\author{J.A. Bergstra \and C.A. Middelburg}
\institute{Informatics Institute, Faculty of Science,
           University of Amsterdam, \\
           Science Park~904, 1098~XH Amsterdam, the Netherlands \\
           \email{J.A.Bergstra@uva.nl,C.A.Middelburg@uva.nl}}

\begin{document}

\maketitle

\begin{abstract}
% 206 %
Inversive meadows are commutative rings with a multiplicative identity
element and a total multiplicative inverse operation satisfying
$0\minv = 0$.
Divisive meadows are inversive meadows with the multiplicative inverse
operation replaced by a division operation.
We give finite equational specifications of the class of all inversive
meadows and the class of all divisive meadows.
It depends on the angle from which they are viewed whether inversive
meadows or divisive meadows must be considered more basic.
We show that inversive and divisive meadows of rational numbers can be
obtained as initial algebras of finite equational specifications.
In the spirit of Peacock's arithmetical algebra, we study variants of
inversive and divisive meadows without an additive identity element
and/or an additive inverse operation.
We propose simple constructions of variants of inversive and divisive
meadows with a partial multiplicative inverse or division operation from
inversive and divisive meadows.
Divisive meadows are more basic if these variants are considered as
well.
We give a simple account of how mathematicians deal with $1 \mdiv 0$, in
which meadows and a customary convention among mathematicians play
prominent parts, and we make plausible that a convincing account,
starting from the popular computer science viewpoint that $1 \mdiv 0$ is
undefined, by means of some logic of partial functions is not
attainable.
\begin{comment}
\begin{classcode}
12E12, 12L05, 12L12, 16E50, 68Q65.
\end{classcode}
\end{comment}
\end{abstract}

\section{Introduction}
\label{sect-introduction}

The primary mathematical structure for measurement and computation is
unquestionably a field.
In~\cite{BT07a}, meadows are proposed as alternatives for fields with a
purely equational axiomatization.
A meadow is a commutative ring with a multiplicative identity element
and a total multiplicative inverse operation satisfying two equations
which imply that the multiplicative inverse of zero is zero.
Thus, meadows are total algebras.
Recently, we found in~\cite{Ono83a} that meadows were already introduced
by Komori~\cite{Kom75a} in a report from 1975, where they go by the name
of \emph{desirable pseudo-fields}.
This finding induced us to propose the name \emph{Komori field} for a
meadow satisfying $0 \neq 1$ and $x \neq 0 \Implies x \mmul x\minv = 1$.
The prime example of Komori fields is the field of rational numbers with
the multiplicative inverse operation made total by imposing that the
multiplicative inverse of zero is zero.

As usual in field theory, the convention to consider $p \mdiv q$ as an
abbreviation for $p \mmul (q\minv)$ was used in subsequent work on
meadows (see e.g.~\cite{BHT09a,BP08a}).
This convention is no longer satisfactory if partial variants of meadows
are considered too, as will be demonstrated in this paper.
That is why we rename meadows into inversive meadows and introduce
divisive meadows.
A divisive meadow is an inversive meadow with the multiplicative inverse
operation replaced by the division operation suggested by the
above-mentioned abbreviation convention.
We give finite equational specifications of the class of all inversive
meadows and the class of all divisive meadows and demonstrate that it
depends on the angle from which they are viewed whether inversive
meadows or divisive meadows must be considered more basic.
Henceforth, we will use the name meadow whenever the distinction between
inversive meadows and divisive meadows is not important.

Peacock introduced in~\cite{Pea1830a} arithmetical algebra as algebra of
numbers where an additive identity element and an additive inverse
operation are not involved.
That is, arithmetical algebra is algebra of positive numbers instead of
algebra of numbers in general (see also~\cite{Kle98a}).
In the spirit of Peacock, we use the name \emph{arithmetical meadow} for
a meadow without an additive identity element and an additive inverse
operation.
Moreover, we use the name \emph{arithmetical meadow with zero} for a
meadow without an additive inverse operation, but with an additive
identity element.
Arithmetical meadows of rational numbers are reminiscent of Peacock's
arithmetical algebra.
We give finite equational specifications of the class of all inversive
arithmetical meadows, the class of all divisive arithmetical meadows,
the class of all inversive arithmetical meadows with zero and the class
of all divisive arithmetical meadows with zero.

The main inversive meadow that we are interested in is the
zero-totalized field of rational numbers, which differs from the field
of rational numbers only in that the multiplicative inverse of zero is
zero.
The main divisive meadow that we are interested in is the zero-totalized
field of rational numbers with the multiplicative inverse operation
replaced by the division operation suggested by the abbreviation
$p \mdiv q$ for $p \mmul (q\minv)$.
We show that these meadows can be obtained as initial algebras of finite
equational specifications.
We also show that arithmeti\-cal meadows of rational numbers and
arithmetical meadows of rational numbers with zero can be obtained as
initial algebras of finite equational specifications.
Arithmetical meadows of rational numbers and arithmetical meadows of
rational numbers with zero provide additional insight in what is yielded
by the presence of an operator for multiplicative inverse (or division)
in a signature.

Partial variants of meadows can be obtained by turning the total
multiplicative inverse or division operation into a partial one.
There is one way in which the total multiplicative inverse operation can
be turned into a partial one, whereas there are two conceivable ways in
which the total division operation can be turned into a partial one.
Therefore, we propose one construction of variants of inversive meadows
with a partial multiplicative inverse operation from inversive meadows
and two constructions of variants of divisive meadows with a partial
division operation from divisive meadows.
We demonstrate that divisive meadows are more basic if those partial
variants of meadows are considered as~well.

We can obtain interesting partial versions of the above-mentioned
meadows of rational numbers, each of which is the initial algebra of a
finite equational specification, by means of the proposed constructions
of partial versions.
This approach fits in with our position that partial algebras should be
made of total ones.
Thus, we can obtain total and partial algebras requiring only equational
logic for total algebras as a tool for their construction.

It is quite usual that neither the division operator nor the
multiplicative inverse operator is included in the signature of number
systems such as the field of rational numbers and the field of real
numbers.
However, the abundant use of the division operator in mathematical
practice makes it very reasonable to include the division operator, or
alternatively the multiplicative inverse operator, in the signature.
It appears that excluding both of them creates more difficulties than
that it solves.
At the least, the problem of division by zero cannot be avoided by
excluding $1 \mdiv 0$ from being written.
We give a simple account of how mathematicians deal with $1 \mdiv 0$ in
mathematical works.
Dominating in this account is the concept of an imperative meadow,
a concept in which a customary convention among mathematicians plays a
prominent part.
We also make plausible that a convincing account, starting from the
usual viewpoint of theoretical computer scientists that $1 \mdiv 0$ is
undefined, by means of some logic of partial functions is not
attainable.

This paper is organized as follows.
First, we go into the background of the work presented in this paper
with the intention to clarify and motivate this work
(Section~\ref{sect-background}) and discuss the main prevailing
viewpoints on the status of $1 \mdiv 0$ in mathematics and theoretical
computer science (Section~\ref{sect-viewpoints-div-by-zero}).
Next, we give equational specifications of the class of all
inversive meadows and the class of all divisive meadows
(Section~\ref{sect-inv-div-meadows}).
After that, we give equational specifications of the arithmetical
variants of those classes (Section~\ref{sect-arith-meadows}) and connect
one of those variants with an arithmetical version of von~Neumann
regular rings (Section~\ref{sect-relation-arith-rings}).
Then, we give equational specifications whose initial algebras are
inversive and divisive meadows of rational numbers
(Section~\ref{sect-meadows-rat}).
After that, we give equational specifications whose initial algebras are
the arithmetical variants of those meadows of rational numbers
(Section~\ref{sect-arith-meadows-rat}).
Following this, we introduce and discuss constructions of partial
variants of meadows from total ones (Section~\ref{sect-partial-meadows})
and constructions of partial variants of arithmetical meadows from total
ones (Section~\ref{sect-partial-arith-meadows}).
Next, we introduce imperative meadows of rational numbers
(Section~\ref{sect-imperative-meadows}) and discuss the convention that
is involved in them (Section~\ref{sect-rel-div-conv}).
After that, we make plausible the inadequacy of logics of partial
functions for a convincing account of how mathematicians deal with
$1 \mdiv 0$ (Section~\ref{sect-inadequacy-LPF}).
Finally, we make some concluding remarks
(Section~\ref{sect-conclusions}).

This paper consolidates material from~\cite{BM09g,BM09h,BM09j}.

\section{Background on the Theory of Meadows}
\label{sect-background}

In this section, we go into the background of the work presented in this
paper with the intention to clarify and motivate this work.

The theory of meadows, see e.g.~\cite{BHT09a,BP08a}, constitutes a
hybrid between the theory of abstract data type and the theory of rings
and fields, more specifically the theory of von Neumann regular
rings~\cite{McC64a,Goo79a} (all fields are von Neumann regular rings).

It is easy to see that each meadow can be reduced to a commutative von
Neumann regular ring with a multiplicative identity element.
Moreover, we know from~\cite{BHT09a} that each commutative von Neumann
regular ring with a multiplicative identity element can be expanded to a
meadow, and that this expansion is unique.
It is easy to show that, if $\morph{\phi}{X}{Y}$ is an epimorphism
between commutative rings with a multiplicative identity element and $X$
is a commutative von Neumann regular ring with a multiplicative
identity element, than:
 (i)~$Y$ is a commutative von Neumann regular ring with a multiplicative
identity element;
(ii)~$\phi$ is also an epimorphism between meadows for the meadows $X'$
and $Y'$ found by means of the unique expansions for $X$ and $Y$,
respectively.

However, there is a difference between commutative von Neumann regular
rings with a multiplicative identity element and meadows: the class of
all meadows is a variety and the class of all commutative von Neumann
regular rings with a multiplicative identity element is not.
In particular, the class of commutative von Neumann regular rings with a
multiplicative identity element is not closed under taking subalgebras
(a property shared by all varieties).
Let $\Rat$ be the ring of rational numbers, and let $\Int$ be its
subalgebra of integers.
Then $\Rat$ is a field and for that reason a commutative von Neumann
regular ring with a multiplicative identity element, but its subalgebra
$\Int$ is not a commutative von Neumann regular ring with a
multiplicative identity element.

In spite of the fact that meadows and commutative von Neumann regular
rings with a multiplicative identity element are so close that no new
mathematics can be expected, there is a difference which matters very
much from the perspective of abstract data type specification.
$\Rat$, the ring of rational numbers, is not a minimal algebra, whereas
$\Ratzi$, the inversive meadow of rational numbers is a minimal algebra.
As such, $\Ratzi$ is amenable to initial algebra specification.
The first initial algebra specification of $\Ratzi$ is given
in~\cite{BT07a} and an improvement due to Hirshfeld is given in the
current paper.
When looking for an initial algebra specification of $\Rat$, adding a
total multiplicative inverse operation satisfying $0\minv = 0$ as an
auxiliary function is the most reasonable solution, assuming that a
proper constructor as an auxiliary function is acceptable.

We see a theory of meadows having two roles:
 (i)~a starting-point of a theory of mathematical data types;
(ii)~an intermediate between algebra and logic.

On investigation of mathematical data types, known countable
mathematical structures will be equipped with operations to obtain
minimal algebras and specification properties of those minimal algebras
will be investigated.%
\begin{comment}
As an example, we mention~\cite{Ber06a}.
\end{comment}
If countable minimal algebras can be classified as either computable,
semi-computable or co-semi-computable, known specification techniques
may be applied (see~\cite{BT95a} for a survey of this matter).
Otherwise data type specification in its original forms cannot be
applied.
Further, one may study $\omega$-completeness of specifications and term
rewriting system related properties.%
\begin{comment}
For the former, we refer to~\cite{BH94a} and, for the latter, we refer
to~\cite{BT95a} for further information.
\end{comment}

It is not a common viewpoint in algebra or in mathematics at large that
giving a name to an operation, which is included in a signature, is a
very significant step by itself.
However, the answer to the notorious question ``what is $1 \mdiv 0$'' is
very sensitive to exactly this matter.
Von Neumann regular rings provide a classical mathematical perspective
on rings and fields, where multiplicative inverse (or division) is only
used when its use is clearly justified and puzzling uses are rejected as
a matter of principle.
Meadows provide a more logical perspective to von Neumann regular rings
in which justified and unjustified use of multiplicative inverse cannot
be easily distinguished beforehand.

\section{Viewpoints on the Status of $1 \mdiv 0$}
\label{sect-viewpoints-div-by-zero}

In this section, we shortly discuss two prevailing viewpoints on the
status of $1 \mdiv 0$ in mathematics and one prevailing viewpoint on
the status of $1 \mdiv 0$ in theoretical computer science.
To our knowledge, the viewpoints in question are the main prevailing
viewpoints.
We take the case of the rational numbers, the case of the real numbers
being essentially the same.

One prevailing viewpoint in mathematics is that $1 \mdiv 0$ has no
meaning because $1$ cannot be divided by $0$.
The argumentation for this viewpoint rests on the fact that there is no
rational number $z$ such that $0 \mmul z = 1$.
Moreover, in mathematics, syntax is not prior to semantics and posing
the question ``what is $1 \mdiv 0$'' is not justified by the mere
existence of $1 \mdiv 0$ as a syntactic object.
Given the fact that there is no rational number that mathematicians
intend to denote by $1 \mdiv 0$, this means that there is no need to
assign a meaning to $1 \mdiv 0$.

Another prevailing viewpoint in mathematics is that the use of
$1 \mdiv 0$ is simply disallowed because the intention to divide $1$ by
$0$ is non-existent in mathematical practice.
This viewpoint can be regarded as a liberal form of the previous one:
the rejection of the possibility that $1 \mdiv 0$ has a meaning is
circumvented by disallowing the use of $1 \mdiv 0$.
Admitting that $1 \mdiv 0$ has a meaning, such as $0$ or ``undefined'',
is consistent with this viewpoint.

The prevailing viewpoint in theoretical computer science is that the
meaning of $1 \mdiv 0$ is ``undefined'' because division is a partial
function.
Division is identified as a partial function because there is no
rational number $z$ such that $0 \mmul z = 1$.
This viewpoint presupposes that the use of $1 \mdiv 0$ should be
allowed, for otherwise assigning a meaning to $1 \mdiv 0$ does not make
sense.
Although this viewpoint is more liberal than the previous one, it is
remote from ordinary mathematical practice.

The first of the two prevailing viewpoints in mathematics discussed
above only leaves room for very informal concepts of expression,
calculation, proof, substitution, etc.
For that reason, we refrain from considering that viewpoint any further
in the current paper.
The prevailing viewpoint in mathematics considered further in this paper
corresponds to the inversive and divisive meadows of rational numbers
together with an imperative about the use of the multiplicative inverse
operator and division operator, respectively.
The prevailing viewpoint in theoretical computer science corresponds to
two of the partial meadows of rational numbers obtained from the
inversive and divisive meadows of rational numbers by constructions
proposed in the current paper.

\section{Inversive Meadows and Divisive Meadows}
\label{sect-inv-div-meadows}

In this section, we give finite equational specifications of the class
of all inversive meadows and the class of all divisive meadows.
In~\cite{BT07a}, inversive meadows were introduced for the first time.
They are further investigated in e.g.~\cite{BHT09a,BP08a,BR08a,BRS09a}.

It appears that, in the sphere of groups, rings and fields, the
qualifications inversive and divisive have only been used by
Yamada~\cite{Yam63a} and Verloren van Themaat~\cite{Ver78a},
respectively.
Our use of these qualifications is in line with theirs.

An inversive meadow is a commutative ring with a multiplicative identity
element and a total multiplicative inverse operation satisfying two
equations which imply that the multiplicative inverse of zero is zero.
A divisive meadow is a commutative ring with a multiplicative identity
element and a total division operation satisfying three equations which
imply that division by zero always yields zero.
Hence, the signature of both inversive and divisive meadows include the
signature of a commutative ring with a multiplicative identity element.

The signature of commutative rings with a multiplicative identity
element consists of the following constants and operators:
\begin{itemize}
\item
the \emph{additive identity} constant $0$;
\item
the \emph{multiplicative identity} constant $1$;
\item
the binary \emph{addition} operator ${} + {}$;
\item
the binary \emph{multiplication} operator ${} \mmul {}$;
\item
the unary \emph{additive inverse} operator $- {}$;
\end{itemize}
The signature of inversive meadows consists of the constants and
operators from the signature of commutative rings with a multiplicative
identity element and in addition:
\begin{itemize}
\item
the unary \emph{multiplicative inverse} operator ${}\minv$.
\end{itemize}
The signature of divisive meadows consists of the constants and
operators from the signature of commutative rings with a multiplicative
identity element and in addition:
\begin{itemize}
\item
the binary \emph{division} operator ${} \mdiv {}$.
\end{itemize}
We write:
\begin{ldispl}
\begin{array}{@{}l@{\;}c@{\;}l@{}}
\sigcr    & \mathrm{for} & \set{0,1,{} + {},{} \mmul {}, - {}}\;,
\\
\sigimd   & \mathrm{for} & \sigcr \union \set{{}\minv}\;,
\\
\sigdmd   & \mathrm{for} & \sigcr \union \set{{} \mdiv {}}\;.
\end{array}
\end{ldispl}

We assume that there are infinitely many variables, including $x$, $y$
and $z$.
Terms are build as usual.
We use infix notation for the binary operators, prefix notation for the
unary operator $- {}$, and postfix notation for the unary operator
${}\minv$.
We use the usual precedence convention to reduce the need for
parentheses.
We introduce subtraction as an abbreviation: $p - q$ abbreviates
$p + (-q)$.
We denote the numerals $0$, $1$, $1 + 1$, $(1 + 1) + 1$, \ldots~by
$\ul{0}$, $\ul{1}$, $\ul{2}$, $\ul{3}$, \ldots~and we use the notation
$p^n$ for exponentiation with a natural number as exponent.
Formally, we define $\ul{n}$ inductively by $\ul{0} = 0$, $\ul{1} = 1$
and $\ul{n + 2} = \ul{n} + 1$ and we define, for each term $p$ over the
signature of inversive meadows or the signature of divisive meadows,
$p^n$ inductively by $p^0 = 1$ and $p^{n+1} = p^n \mmul p$.

The constants and operators from the signatures of inversive meadows and
divisive meadows are adopted from rational arithmetic, which gives an
appropriate intuition about these constants and operators.
The set of all terms over the signature of inversive meadows constitutes
the \emph{inversive meadow notation} and
the set of all terms over the signature of divisive meadows constitutes
the \emph{divisive meadow notation}.

A commutative ring with a multiplicative identity element is an algebra
over the signature $\sigcr$ that satisfies the equations given in
Table~\ref{eqns-commutative-ring}.
\begin{table}[!t]
\caption
{Axioms of a commutative ring with a multiplicative identity element}
\label{eqns-commutative-ring}
\begin{eqntbl}
\begin{eqncol}
(x + y) + z = x + (y + z)                                             \\
x + y = y + x                                                         \\
x + 0 = x                                                             \\
x + (-x) = 0
\end{eqncol}
\qquad\quad
\begin{eqncol}
(x \mmul y) \mmul z = x \mmul (y \mmul z)                             \\
x \mmul y = y \mmul x                                                 \\
x \mmul 1 = x                                                         \\
x \mmul (y + z) = x \mmul y + x \mmul z
\end{eqncol}
\end{eqntbl}
\end{table}
An inversive meadow is an algebra over the signature $\sigimd$ that
satisfies the equations given in Tables~\ref{eqns-commutative-ring}
and~\ref{eqns-add-inversive-meadow}.%
\begin{table}[!t]
\caption{Additional axioms for an inversive meadow}
\label{eqns-add-inversive-meadow}
\begin{eqntbl}
\begin{eqncol}
{(x\minv)}\minv = x                                                   \\
x \mmul (x \mmul x\minv) = x
\end{eqncol}
\end{eqntbl}
\end{table}
A divisive meadow is an algebra over the signature $\sigdmd$ that
satisfies the equations given in Tables~\ref{eqns-commutative-ring}
and~\ref{eqns-add-divisive-meadow}.%
\begin{table}[!t]
\caption{Additional axioms for a divisive meadow}
\label{eqns-add-divisive-meadow}
\begin{eqntbl}
\begin{eqncol}
1 \mdiv (1 \mdiv x) = x                                               \\
(x \mmul x) \mdiv x = x                                               \\
x \mdiv y = x \mmul (1 \mdiv y)
\end{eqncol}
\end{eqntbl}
\end{table}
We write:
\begin{ldispl}
\begin{array}{@{}l@{\;}c@{\;}l@{}}
\eqnscr  &
\multicolumn{2}{@{}l@{}}
 {\mathrm{for\; the\; set\; of\; all\; equations\; in\; Table\;
          \ref{eqns-commutative-ring}}\;,}
\\
\eqnsinv  &
\multicolumn{2}{@{}l@{}}
 {\mathrm{for\; the\; set\; of\; all\; equations\; in\; Table\;
          \ref{eqns-add-inversive-meadow}}\;,}
\\
\eqnsdiv  &
\multicolumn{2}{@{}l@{}}
 {\mathrm{for\; the\; set\; of\; all\; equations\; in\; Table\;
          \ref{eqns-add-divisive-meadow}}\;,}
\\
\eqnsimd & \mathrm{for} & \eqnscr \union \eqnsinv\;,
\\
\eqnsdmd & \mathrm{for} & \eqnscr \union \eqnsdiv\;.
\end{array}
\end{ldispl}

The equation ${(x\minv)}\minv = x$ is called the reflexivity equation
and the equation $x \mmul (x \mmul x\minv) = x$ is called the restricted
inverse equation.
The first two equations in Table~\ref{eqns-add-divisive-meadow} are the
obvious counterparts of the reflexivity equation and restricted inverse
equation in divisive meadows.
The equation $0\minv = 0$ is derivable from the equations $\eqnsimd$.
The equation $x \mdiv 0 = 0$ is derivable from the equations $\eqnsdmd$.
The equation $1 \mdiv 0 = 0$ can be derived without using the equation
$x \mdiv y = x \mmul (1 \mdiv y)$, and then the latter equation can be
applied to derive the equation $x \mdiv 0 = 0$.

The advantage of working with a total multiplicative inverse operation
or total division operation lies in the fact that conditions like
$x \neq 0$ in $x \neq 0 \Implies x \mmul x\minv = 1$ or
$x \neq 0 \Implies x \mmul (1 \mdiv x) = 1$ are not needed to guarantee
meaning.

In~\cite{BL02a}, projection semantics is proposed as an approach to
define the meaning of programs.
Projection semantics explains the meaning of programs in terms of known
programs instead of in terms of more or less sophisticated mathematical
objects.
Here, we transpose this approach to the current setting to demonstrate
that it depends on the angle from which they are viewed whether
inversive meadows or divisive meadows must be considered more basic.

We can explain the meaning of the terms over the signature of divisive
meadows by means of a projection $\dmnimn$ from the divisive meadow
notation to the inversive meadow notation.
This projection is defined as follows:
\begin{sldispl}
\dmnimn(x) = x\;, \\
\dmnimn(0) = 0\;, \\
\dmnimn(1) = 1\;, \\
\dmnimn(p + q) = \dmnimn(p) + \dmnimn(q)\;, \\
\dmnimn(p \mmul q) = \dmnimn(p) \mmul \dmnimn(q)\;, \\
\dmnimn(- p) = - \dmnimn(p)\;, \\
\dmnimn(p \mdiv q) = \dmnimn(p) \mmul (\dmnimn(q)\minv)\;.
\end{sldispl}
The projection $\dmnimn$ supports an interpretation of the theory of
divisive meadows in the theory of inversive meadows: for each equation
$p = q$ derivable from the equations $\eqnsdmd$, the equation
$\dmnimn(p) = \dmnimn(q)$ is derivable from the equations $\eqnsimd$.%
\footnote
{For the notion of a translation that supports a theory interpretation,
 see e.g.~\cite{Vis06a}.}
Therefore the projection $\dmnimn$ determines a mapping from divisive
meadows to inversive meadows.

We can also explain the meaning of the terms over the signature of
inversive meadows by means of a projection $\imndmn$ from the inversive
meadow notation to the divisive meadow notation.
This projection is defined as follows:
\pagebreak[2]
\begin{sldispl}
\imndmn(x) = x\;, \\
\imndmn(0) = 0\;, \\
\imndmn(1) = 1\;, \\
\imndmn(p + q) = \imndmn(p) + \imndmn(q)\;, \\
\imndmn(p \mmul q) = \imndmn(p) \mmul \imndmn(q)\;, \\
\imndmn(- p) = - \imndmn(p)\;, \\
\imndmn(p\minv) = 1 \mdiv \imndmn(p)\;.
\end{sldispl}
The projection $\imndmn$ supports an interpretation of the theory of
inversive meadows in the theory of divisive meadows: for each equation
$p = q$ derivable from the equations $\eqnsimd$, the equation
$\imndmn(p) = \imndmn(q)$ is derivable from the equations $\eqnsdmd$.
Therefore the projection $\imndmn$ determines a mapping from inversive
meadows to divisive meadows.

Given the finite equational specification of the class of all inversive
meadows, we can easily give a modular specification of the class of all
divisive meadows using module algebra~\cite{BHK88a}.
In Appendix~\ref{sect-module-algebra}, we give the modular specification
in question and show that the equational theory associated with it is
the same as the equational theory associated with the equational
specification of the class of all divisive meadows.

A \emph{non-trivial inversive meadow} is an inversive meadow that
satisfies the \emph{separation axiom} $0 \neq 1$.
An \emph{inversive cancellation meadow} is an inversive meadow that
satisfies the \emph{cancellation axiom}
$x \neq 0 \And x \mmul y = x \mmul z \Implies y = z$,
or equivalently, the \emph{general inverse law}
$x \neq 0 \Implies x \mmul x\minv = 1$.
An \emph{inversive Komori field} is an inversive meadow that satisfies
the separation axiom and the cancellation axiom.
A \emph{non-trivial divisive meadow} is an divisive meadow that
satisfies the separation axiom.
A \emph{divisive cancellation meadow} is an divisive meadow that
satisfies the cancellation axiom.
A \emph{divisive Komori field} is an divisive meadow that satisfies
the separation axiom and the cancellation axiom.

An important property of inversive Komori fields is the following:
$0 \mmul (0\minv) = 0$, whereas $x \mmul (x\minv) = 1$ for $x \neq 0$.
An important property of divisive Komori fields is the following: $0
\mdiv 0 = 0$, whereas $x \mdiv x = 1$ for $x \neq 0$.

The inversive Komori field that we are most interested in is $\Ratzi$,
the inversive Komori field of rational numbers.
The divisive Komori field that we are most interested in is $\Ratzd$,
the divisive Komori field of rational numbers.
In Section~\ref{sect-meadows-rat}, both $\Ratzi$ and $\Ratzd$ will be
obtained by means of the well-known initial algebra construction.
$\Ratzi$ differs from the field of rational numbers only in that the
multiplicative inverse of zero is zero.
$\Ratzd$ differs from $\Ratzi$ only in that the multiplicative inverse
operation is replaced by a division operation such that
$x \mdiv y = x \mmul y\minv$.

A reduced divisive meadow is an algebra over the signature
$\set{1,{} - {},{} \mdiv {}}$ that satisfies the equations given in
Table~\ref{eqns-reduced-divisive-meadow}.
\begin{table}[!t]
\caption{Axioms of a reduced divisive meadow}
\label{eqns-reduced-divisive-meadow}
\begin{eqntbl}
\begin{eqncol}
(x - ((1 - 1) - y)) - ((1 - 1) - z) =
x - ((1 - 1) - (y - ((1 - 1) - z)))                                   \\
x - ((1 - 1) - y) = y - ((1 - 1) - x)                                 \\
x - (1 - 1) = x                                                       \\
x - x = 1 - 1                                                         \\
(x \mdiv (1 \mdiv y)) \mdiv (1 \mdiv z) =
x \mdiv (1 \mdiv (y \mdiv (1 \mdiv z)))                               \\
x \mdiv (1 \mdiv y) = y \mdiv (1 \mdiv x)                             \\
x \mdiv 1 = x                                                         \\
x \mdiv (1 \mdiv (y - ((1 - 1) - z))) =
x \mdiv (1 \mdiv y) - ((1 - 1) - (x \mdiv (1 \mdiv z)))               \\
( x \mdiv (1 \mdiv x)) \mdiv x = x
\end{eqncol}
\end{eqntbl}
\end{table}
We can explain the meaning of the terms over the signature of inversive
meadows by means of a projection $\imnrdmn$ to  terms over the signature
of reduced divisive meadows.
This projection is defined as follows:
\pagebreak[2]
\begin{sldispl}
\imnrdmn(x) = x\;, \\
\imnrdmn(0) = 1 - 1\;, \\
\imnrdmn(1) = 1\;, \\
\imnrdmn(p + q) = \imnrdmn(p) - ((1 - 1) - \imnrdmn(q))\;, \\
\imnrdmn(p \mmul q) = \imnrdmn(p) \mdiv (1 \mdiv \imnrdmn(q))\;, \\
\imnrdmn(- p) = (1 - 1) - \imnrdmn(p)\;, \\
\imnrdmn(p\minv) = 1 \mdiv \imnrdmn(p)\;.
\end{sldispl}
The projection $\imnrdmn$ supports an interpretation of the theory of
inversive meadows in the theory of reduced divisive meadows.

The following are some outstanding questions with regard to inversive
meadows, divisive meadows, and reduced divisive meadows:
\begin{enumerate}
\item
Do there exist equational specifications of the class of all inversive
meadows, the class of all divisive meadows, and the class of all reduced
divisive meadows with less than $10$ equations, $11$ equations, and $9$
equations, respectively?
\item
Can the number of binary operators needed to explain the meaning of the
terms over the signature of inversive meadows be reduced from two to
one?
\end{enumerate}

\section{Arithmetical Meadows}
\label{sect-arith-meadows}

In this section, we give finite equational specifications of the class
of all inversive arithmetical meadows, the class of all divisive
arithmetical meadows, the class of all inversive arithmetical meadows
with zero and the class of all divisive arithmetical meadows with zero.

The signatures of inversive and divisive arithmetical meadows with zero
are the signatures of inversive and divisive meadows with the additive
inverse operator $- {}$ removed.
The signatures of inversive and divisive arithmetical meadows are the
signatures of inversive and divisive arithmetical meadows with zero with
the additive identity constant $0$ removed.
We write:
\begin{ldispl}
\begin{array}{@{}l@{\;}c@{\;}l@{}}
\sigiamdz & \mathrm{for} & \sigimd \diff \set{- {}}\;,
\\
\sigdamdz & \mathrm{for} & \sigdmd \diff \set{- {}}\;,
\\
\sigiamd  & \mathrm{for} & \sigiamdz \diff \set{0}\;,
\\
\sigdamd  & \mathrm{for} & \sigdamdz \diff \set{0}\;.
\end{array}
\end{ldispl}
Moreover, we write:
\begin{ldispl}
\begin{array}{@{}l@{\;}c@{\;}l@{}}
\eqnsacrz & \mathrm{for} &
\eqnscr \diff \set{x + (-x) = 0}\;,
\\
\eqnsacr  & \mathrm{for} &
\eqnsacrz \diff \set{x + 0 = x}\;.
\end{array}
\end{ldispl}
The equations in $\eqnsacrz$ are the equations from $\eqnscr$ in which
the additive inverse operator $- {}$ does not occur.
The equations in $\eqnsacr$ are the equations from $\eqnsacrz$ in which
the additive identity constant $0$ does not occur.

An \emph{inversive arithmetical meadow} is an algebra over the
signature $\sigiamd$ that satisfies the equations $\eqnsacr$ and the
equation $x \mmul x\minv = 1$.
A \emph{divisive arithmetical meadow} is an algebra over the signature
$\sigdamd$ that satisfies the equations $\eqnsacr$ and the equation
$x \mdiv x = 1$.
An \emph{inversive arithmetical meadow with zero} is an algebra over the
signature $\sigiamdz$ that satisfies the equations $\eqnsacrz$ and the
equations $\eqnsinv$.
A \emph{divisive arithmetical meadow with zero} is an algebra over the
signature $\sigdamdz$ that satisfies the equations $\eqnsacrz$ and the
equations $\eqnsdiv$.
We write:
\begin{ldispl}
\begin{array}{@{}l@{\;}c@{\;}l@{}}
\eqnsiamd & \mathrm{for} &
\eqnsacr \union \set{x \mmul x\minv = 1}\;,
\\
\eqnsdamd & \mathrm{for} &
\eqnsacr \union \set{x \mdiv x = 1}\;,
\\
\eqnsiamdz & \mathrm{for} &
\eqnsacrz \union \eqnsinv\;,
\\
\eqnsdamdz & \mathrm{for} &
\eqnsacrz \union \eqnsdiv\;.
\end{array}
\end{ldispl}

The arithmetical meadows that we are most interested in are the
arithmetical meadows of rational numbers and the arithmetical meadow
with zero that we are most interested in are the arithmetical meadow
of rational numbers with zero.
In Section~\ref{sect-arith-meadows-rat}, those arithmetical meadows will
be obtained by means of the well-known initial algebra construction.
The following lemmas about arithmetical meadows and arithmetical
meadows with zero will be used in Section~\ref{sect-arith-meadows-rat}.
\begin{lemma}
\label{lemma-numerals}
For all $n,m \in \Nat \diff \set{0}$, we have that
$\eqnsacr \deriv \ul{n + m} = \ul{n} + \ul{m}$ and
$\eqnsacr \deriv \ul{n \mmul m} = \ul{n} \mmul \ul{m}$.
\end{lemma}
\begin{proof}
The fact that $\ul{n + m} = \ul{n} + \ul{m}$ is derivable from
$\eqnsacr$ is easily proved by induction on $n$.
The basis step is trivial.
The inductive step goes as follows:
$\ul{(n + 1) + m} = \ul{(n + m) + 1} = \ul{n + m} + 1 =
 \ul{n} + \ul{m} + 1 = \ul{n} + 1 + \ul{m} = \ul{n + 1} + \ul{m}$.
The fact that $\ul{n \mmul m} = \ul{n} \mmul \ul{m}$ is derivable from
$\eqnsacr$ is easily proved by induction on $n$, using that
$\ul{n + m} = \ul{n} + \ul{m}$ is derivable from $\eqnsacr$.
The basis step is trivial.
The inductive step goes as follows:
$\ul{(n + 1) \mmul m} = \ul{n \mmul m + 1 \mmul m} =
 \ul{n \mmul m} + \ul{1 \mmul m} =
 \ul{n} \mmul \ul{m} + \ul{1} \mmul \ul{m} =
 (\ul{n} + \ul{1}) \mmul \ul{m} = \ul{n + 1} \mmul \ul{m}$.
\qed
\end{proof}
\begin{lemma}
\label{lemma-mult-inv-ratia}
We have $\eqnsiamd \deriv (x\minv)\minv = x$ and
$\eqnsiamd \deriv (x \mmul y)\minv = x\minv \mmul y\minv$.
\end{lemma}
\begin{proof}
We derive $(x\minv)\minv = x$ from $\eqnsiamd$ as
follows:
$(x\minv)\minv = 1 \mmul (x\minv)\minv =
 (x \mmul x\minv) \mmul (x\minv)\minv =
 x \mmul (x\minv \mmul (x\minv)\minv) =
 x \mmul 1 = x$.
We derive $(x \mmul y)\minv = x\minv \mmul y\minv$ from $\eqnsiamd$ as
follows:
$(x \mmul y)\minv = 1 \mmul (1 \mmul (x \mmul y)\minv) =
 (x \mmul x\minv) \mmul ((y \mmul y\minv) \mmul (x \mmul y)\minv) =
 (x\minv \mmul y\minv) \mmul ((x \mmul y) \mmul (x \mmul y)\minv) =
 (x\minv \mmul y\minv) \mmul 1 = x\minv \mmul y\minv$.
\qed
\end{proof}
\begin{lemma}
\label{lemma-zero-ratiaz}
We have $\eqnsiamdz \deriv 0 \mmul x = 0$ and
$\eqnsiamdz \deriv 0\minv = 0$.
\end{lemma}
\begin{proof}
Firstly, we derive $x + y = x \Implies y = 0$ from $\eqnsiamdz$ as
follows:
$x + y = x \Implies 0 + y = 0 \Implies y + 0 = 0 \Implies y = 0$.
Secondly, we derive $x + 0 \mmul x = x$ from $\eqnsiamdz$ as follows:
$x + 0 \mmul x = x \mmul 1 + 0 \mmul x = 1 \mmul x + 0 \mmul x =
 (1 + 0) \mmul x = 1 \mmul x = x \mmul 1 = x$.
From $x + y = x \Implies y = 0$ and $x + 0 \mmul x = x$, it follows that
$0 \mmul x = 0$.
We derive $0\minv = 0$ from $\eqnsiamdz$ as follows:
$0\minv = 0\minv \mmul (0\minv \mmul (0\minv)\minv) =
 (0\minv)\minv \mmul (0\minv \mmul 0\minv) =
 0 \mmul (0\minv \mmul 0\minv) = 0$.
\qed
\end{proof}
\begin{lemma}
\label{lemma-mult-inv-ratiaz}
We have $\eqnsiamdz \deriv (x \mmul y)\minv = x\minv \mmul y\minv$.
\end{lemma}
\begin{proof}
Proposition~2.8 from~\cite{BHT09a} states that
$(x \mmul y)\minv = x\minv \mmul y\minv$ is derivable from
$\eqnsiamdz \union \set{x + 0 = x, x + (-x) = 0}$.
The proof of this proposition given in~\cite{BHT09a} goes through
because no use is made of the equations $x + 0 = x$ and $x + (-x) = 0$.
\qed
\end{proof}
\begin{lemma}
\label{lemma-zero-elim-ratiaz}
For each $\sigiamdz$-term $t$, either $\eqnsiamdz \deriv t = 0$ or there
exists a $\sigiamd$-term $t'$ such that $\eqnsiamdz \deriv t = t'$.
\end{lemma}
\begin{proof}
The proof is easy by induction on the structure of $t$, using
Lemma~\ref{lemma-zero-ratiaz}.
\qed
\end{proof}

\section{Arithmetical Meadows and Regular Arithmetical Rings}
\label{sect-relation-arith-rings}

We can define commutative arithmetical rings with a multiplicative
identity element in the same vein as arithmetical meadows.
Moreover, we can define commutative von Neumann regular arithmetical
rings with a multiplicative identity element as commutative arithmetical
rings with a multiplicative identity element satisfying the regularity
condition $\Forall{x}{\Exists{y}{x \mmul (x \mmul y) = x}}$.

The following theorem states that commutative von Neumann regular
arithmetical rings with a multiplicative identity element are related to
inversive arithmetical meadows like commutative von Neumann regular
rings with a multiplicative identity element are related to inversive
meadows.
\begin{theorem}
Each commutative von Neumann regular arithmetical ring with a
multiplicative identity element can be expanded to an inversive
arithmetical meadow, and this expansion is unique.
\end{theorem}
\begin{proof}
Lemma~2.11 from~\cite{BHT09a} states that each commutative von Neumann
regular ring with a multiplicative identity element can be expanded to
an inversive meadow, and this expansion is unique.
The only use that is made of the equations $x + 0 = x$ and
$x + (-x) = 0$ in the proof of this lemma given in~\cite{BHT09a}
originates from the proof of another lemma that is used in the proof.
However, the latter lemma, Lemma~2.12 from~\cite{BHT09a}, concerns the
same property as Proposition~2.3 from~\cite{BRS09a} and in the proof of
this proposition given in~\cite{BRS09a} no use is made of the equations
$x + 0 = x$ and $x + (-x) = 0$.
Hence, there is an alternative proof of Lemma~2.11 from~\cite{BHT09a}
that goes through for the arithmetical case.
\qed
\end{proof}

We can also define commutative arithmetical rings with additive and
multiplicative identities and commutative von Neumann regular
arithmetical rings with additive and multiplicative identities in the
obvious way.
We also have that commutative von Neumann regular arithmetical rings
with additive and multiplicative identities are related to inversive
arithmetical meadows with zero like commutative von Neumann regular
rings with a multiplicative identity element are related to inversive
meadows.

\section{Meadows of Rational Numbers}
\label{sect-meadows-rat}

In this section, we obtain inversive and divisive meadows of rational
numbers as initial algebras of finite equational specifications.
Moreover, we prove that the inversive meadow in question differs from the
field of rational numbers only in that the multiplicative inverse of
zero is zero.
As usual, we write $I(\Sigma,E)$ for the initial algebra among the
algebras over the signature $\Sigma$ that satisfy the equations $E$
(see e.g.~\cite{BT87a}).

The inversive meadow that we are interested in is $\Ratzi$, the
inversive meadow of rational numbers:
\begin{ldispl}
\Ratzi =
I(\sigimd,\eqnsimd \union
  \set{(1 + x^2 + y^2) \mmul (1 + x^2 + y^2)\minv = 1})\;.
\end{ldispl}
The divisive meadow that we are interested in is $\Ratzd$, the
divisive meadow of rational numbers:
\begin{ldispl}
\Ratzd =
I(\sigdmd,\eqnsdmd \union
  \set{(1 + x^2 + y^2) \mdiv (1 + x^2 + y^2) = 1})\;.
\end{ldispl}
$\Ratzd$ differs from $\Ratzi$ only in that the multiplicative inverse
operation is replaced by a division operation in conformity with the
projection $\imndmn$ defined in Section~\ref{sect-inv-div-meadows}.

To prove that $\Ratzi$ differs from the field of rational numbers only
in that the multiplicative inverse of zero is zero, we need some
auxiliary results.

\begin{lemma}
\label{lemma-primes}
Let $p$ be a prime number.
Then for each $u \in \bbZ_p$, there exists $v,w \in \bbZ_p$ such that
$u = v^2 + w^2$.
\end{lemma}
\begin{proof}
The case where $p = 2$ is trivial.
In the case where $p \neq 2$, $p$ is odd, say $2 \mmul n + 1$.
Let $S$ be the set
$\set{u \in \bbZ_p \where \Exists{v \in \bbZ_p}{u = v^2}}$,
and let $c \in \bbZ_p$ be such that $c \notin S$.
Because $0 \in \bbZ_p$ and each element of $S$ has at most two roots,
we have $|S| \geq n + 1$.
For each $u \in c \mmul S$, $u = 0$ or $u \notin S$, as $u \neq 0$ and
$u \in S$ only if $c \in S$.
Because $c \mmul u \neq c \mmul v$ for each $u,v \in S$ with $u \neq v$,
we have $|c \mmul S| \geq n + 1$.
It follows that $S \union c \mmul S = \bbZ_p$ and
$S \inter c \mmul S = \set{0}$.
This implies that
$c \mmul S =
 \set{u \in \bbZ_p \where \Forall{v \in \bbZ_p}{u \neq v^2}} \union
 \set{0}$.
Hence, for each $u \in \bbZ_p$ with $u \notin S$, there exists an
$v \in \bbZ_p$ such that $u = c \mmul v^2$.
The set $S$ is not closed under sums, as $1 \in S$, and every element of
$\bbZ_p$ is a sum of ones.
This implies that there exist $u,v \in \bbZ_p$ such that
$u^2 + v^2 \notin S$.
Let $a,b \in \bbZ_p$ be such that $a^2 + b^2 \notin S$, and
take $a^2 + b^2$ for $c$.
Then for each $u \in \bbZ_p$ with $u \notin S$, there exists an
$v \in \bbZ$ such that $u = (a^2 + b^2) \mmul v^2$.
Because $(a^2 + b^2) \mmul v^2 = (a \mmul v)^2 + (b \mmul v)^2$, we have
that, for each $u \in \bbZ_p$ with $u \notin S$, there exist
$v,w \in \bbZ$ such that $u = v^2 + w^2$.
Because $u \in S$ iff $u = v^2 + 0^2$ for some $v \in \bbZ_p$, we have
that, for each $u \in \bbZ_p$ with $u \in S$, there exist $v,w \in \bbZ$
such that $u = v^2 + w^2$.
\qed
\end{proof}

\begin{corollary}
\label{corollary-primes}
Let $p$ be a prime number.
Then there exists $u,v,w \in \Nat$ such that
$w \mmul p = u^2 + v^2 + 1$.
\end{corollary}
\begin{proof}
By Lemma~\ref{lemma-primes}, there exist $u,v \in \bbZ_p$ such that
$-1 = u^2 + v^2$.
Let $a,b \in \bbZ_p$ be such that $-1 = a^2 + b^2$.
Then $a^2 + b^2 + 1$ is a multiple of $p$ in $\Nat$.
Hence, there exists $u,v,w \in \Nat$ such that
$w \mmul p = u^2 + v^2 + 1$.
\qed
\end{proof}

\begin{theorem}
\label{theorem-init-alg-spec-ratzi}
$\Ratzi$ is the zero-totalized field of rational numbers, i.e.\ the
$\sigimd$-algebra that differs from the field of rational numbers only
in that $0\minv = 0$.
\end{theorem}
\begin{proof}
From the proof of Theorem~3.6 from~\cite{BT07a}, we already know that,
for each set $E'$ of $\sigimd$-equations valid in the zero-totalized
field of rational numbers, $I(\sigimd,\eqnsimd \union E')$ is the
zero-totalized field of rational numbers if it follows from
$\eqnsimd \union E'$ that $\ul{u}$ has a multiplicative inverse for each
$u \in \Nat \diff \set{0}$.
Because $1 + x^2 + y^2 \neq 0$, we have that
$(1 + x^2 + y^2) \mmul (1 + x^2 + y^2)\minv = 1$ is valid in the
zero-totalized field of rational numbers.
So it remains to be proved that $\ul{u}$ has a multiplicative inverse
for each $u \in \Nat \diff \set{0}$.

Let $p$ be a prime number.
Then, by Corollary~\ref{corollary-primes}, there exist $u,v,w \in \Nat$
such that $w \mmul p = u^2 + v^2 + 1$.
Let $m,a,b \in \Nat$ be such that $m \mmul p = a^2 + b^2 + 1$.
As a corollary of Lemma~\ref{lemma-numerals}, we have
$\ul{u + v} = \ul{u} + \ul{v}$ and
$\ul{u \mmul v} = \ul{u} \mmul \ul{v}$ for all $u,v \in \Nat$.
\linebreak[2]
It follows that $\ul{m} \mmul \ul{p} = \ul{a}^2 + \ul{b}^2 + \ul{1}$.
Because $(1 + x^2 + y^2) \mmul (1 + x^2 + y^2)\minv = 1$,\linebreak[2]
we have $(\ul{m} \mmul \ul{p}) \mmul (\ul{m} \mmul \ul{p})^{-1} = 1$.
This implies that $\ul{m} \mmul (\ul{m} \mmul \ul{p})^{-1}$ is the
multiplicative inverse of $\ul{p}$.
Hence, $\ul{u}$ has a multiplicative inverse for each
$u \in \Nat \diff \set{0}$ that is a prime number.
Let $c \in \Nat \diff \set{0}$.
Then $c$ is the product of finitely many prime numbers, say
$p_1 \mmul {} \cdots {} \mmul p_n$.
Because
$(\ul{p_1} \mmul {} \cdots {} \mmul \ul{p_n})\minv =
 \ul{p_1}\minv \mmul {} \cdots {} \mmul \ul{p_n}\minv$
(see e.g.\ Proposition~2.8 in~\cite{BHT09a}) and
$\ul{c} = \ul{p_1} \mmul {} \cdots {} \mmul \ul{p_n}$, we have that
$\ul{p_1}\minv \mmul {} \cdots {} \mmul \ul{p_n}\minv$
is the multiplicative inverse of $\ul{c}$.
Hence, $\ul{u}$ has a multiplicative inverse for each
$u \in \Nat \diff \set{0}$.
\qed
\end{proof}
Lemma~\ref{lemma-primes}, Corollary~\ref{corollary-primes}, and
Theorem~\ref{theorem-init-alg-spec-ratzi} come from Hirshfeld
(personal communication, 31 January 2009).
Lemma~\ref{lemma-primes} is a folk theorem in the area of field theory,
but we could not find a proof of it in the literature.

We remark that in~\cite{BT07a}, the initial algebra specification of
$\Ratzi$ is obtained by adding the equation
$(1+x^2+y^2+z^2+w^2) \mmul (1+x^2+y^2+z^2+w^2)\minv = 1$
instead of the equation $(1+x^2+y^2) \mmul (1+x^2+y^2)\minv = 1$ to
$\eqnsimd$.
In other words, in the current paper, we have reduced the number of
squares needed in the equation added to $\eqnsimd$ from $4$ to $2$.
In~\cite{BHT09b}, it is shown that the number of squares cannot be
reduced to $1$.

\section{Arithmetical Meadows of Rational Numbers}
\label{sect-arith-meadows-rat}

In this section, we obtain inverse and divisive arithmetical meadows of
rational numbers and inverse and divisive arithmetical meadows of
rational numbers with zero as initial algebras of finite equational
specifications.
Moreover, we prove that the inversive meadows in question are
subalgebras of reducts of the inversive meadow of rational numbers and
some results concerning the decidability of derivability from the
equational specifications concerned.

$\Ratia$, the inversive arithmetical meadow of rational numbers, is
defined as follows:
\begin{ldispl}
\Ratia = I(\sigiamd,\eqnsiamd)\;.
\end{ldispl}
$\Ratda$, the divisive arithmetical meadow of rational numbers, is
defined as follows:
\begin{ldispl}
\Ratda = I(\sigdamd,\eqnsdamd)\;.
\end{ldispl}
Notice that $\Ratia$ and $\Ratda$ are the initial algebras in the class
of inversive arithmetical meadows and the class of divisive arithmetical
meadows, respectively.

$\Ratia$ is a subalgebra of a reduct of $\Ratzi$.
\begin{theorem}
\label{theorem-init-alg-spec-ratia}
$\Ratia$ is the subalgebra of the $\sigiamd$-reduct of $\Ratzi$ whose
domain is the set of all positive rational numbers.
\end{theorem}
\begin{proof}
Like in the case of Theorem~3.1 from~\cite{BT07a}, it is sufficient to
prove that, for each closed term $t$ over the signature $\sigiamd$,
there exists a unique term $t'$ in the set
\begin{ldispl}
\set{\ul{n} \mmul \ul{m}\minv \where
n,m \in \Nat \diff \set{0} \And \nm{gcd}(n,m) = 1}
\end{ldispl}
such that $\eqnsiamd \deriv t = t'$.
Like in the case of Theorem~3.1 from~\cite{BT07a}, this is proved by
induction on the structure of $t$, using Lemmas~\ref{lemma-numerals}
and~\ref{lemma-mult-inv-ratia}.
The proof is similar, but simpler owing to:
(i)~the absence of terms of the forms $0$ and $-t'$;
(ii)~the absence of terms of the forms $\ul{0}$ and
$- (\ul{n} \mmul \ul{m}\minv)$ among the terms that exist by the
induction hypothesis;
(iii)~the presence of the axiom $x \mmul x\minv = 1$.
\qed
\end{proof}
The fact that $\Ratda$ is a subalgebra of a reduct of $\Ratzd$ is proved
similarly.

Derivability of equations from the equations of the initial algebra
specification of $\Ratia$ is decidable.
\begin{theorem}
\label{theorem-decidability-ratia}
For all $\sigiamd$-terms $t$ and $t'$, it is decidable whether
$\eqnsiamd \deriv t = t'$.
\end{theorem}
\begin{proof}
For each $\sigiamd$-term $r$, there exist $\sigiamd$-terms $r_1$ and
$r_2$ in which the multiplicative inverse operator do not occur such
that $\eqnsiamd \deriv r = r_1 \mmul r_2\minv$.
The proof of this fact is easy by induction on the structure of $r$,
using Lemma~\ref{lemma-mult-inv-ratia}.
Inspection of the proof yields that there is an effective way to find
witnessing terms.

For each closed $\sigiamd$-term $r$ in which the multiplicative inverse
operator does not occur there exists a $k \in \Nat \diff \set{0}$, such
that $\eqnsiamd \deriv r = \ul{k}$.
The proof of this fact is easy by induction on the structure of $r$.
Moreover, for each $\sigiamd$-term $r$ in which the multiplicative
inverse operator does not occur there exists a $\sigiamd$-term $r'$ of
the form
$\sum_{i_1=1}^{n_1} \ldots \sum_{i_m=1}^{n_m}
  \ul{k_{i_1 \ldots i_m}} \mmul
  x_1^{i_1} \mmul {} \cdots {} \mmul x_m^{i_m}$,
where $k_{i_1 \ldots i_m} \in \Nat \diff \set{0}$ for each
$i_1 \in [1,n_1]$, \ldots, $i_m \in [1,n_m]$ and $x_1,\ldots,x_m$ are
variables, such that $\eqnsiamd \deriv r = r'$.
The proof of this fact is easy by induction on the structure of $r$,
using the previous fact.
Inspection of the proof yields that there is an effective way to find
a witnessing term.
Terms of the form described above are polynomials in several variables
with positive integer coefficients.

Let $t_1$, $t_2$, $t'_1$, $t'_2$ be $\sigiamd$-terms in which the
multiplicative inverse operator do not occur such that
$\eqnsiamd \deriv t = t_1 \mmul {t_2}\minv$ and
$\eqnsiamd \deriv t' = t'_1 \mmul {t'_2}\minv$.
Moreover, let $s$ and $s'$ be $\sigiamd$-terms of the form
$\sum_{i_1=1}^{n_1} \ldots \sum_{i_m=1}^{n_m}
  \ul{k_{i_1 \ldots i_m}} \mmul
  x_1^{i_1} \mmul {} \cdots {} \mmul x_m^{i_m}$,
where $k_{i_1 \ldots i_m} \in \Nat \diff \set{0}$ for each
$i_1 \in [1,n_1]$, \ldots, $i_m \in [1,n_m]$ and $x_1,\ldots,x_m$ are
variables, such that $\eqnsiamd \deriv t_1 \mmul t'_2 = s$ and
$\eqnsiamd \deriv t'_1 \mmul t_2 = s'$.
We have that
$\eqnsiamd \deriv t = t'$ iff
$\eqnsiamd \deriv t_1 \mmul {t_2}\minv = t'_1 \mmul {t'_2}\minv$ iff
$\eqnsiamd \deriv t_1 \mmul t'_2 = t'_1 \mmul t_2$ iff
$\eqnsiamd \deriv s = s'$.
Moreover, we have that $\eqnsiamd \deriv s = s'$ only if $s$ and $s'$
denote the same function on positive real numbers in the inversive
arithmetical meadow of positive real numbers.
The latter is decidable because polynomials in several variables with
positive integer coefficients denote the same function on positive real
numbers in the inversive arithmetical meadow of positive real numbers
only if they are syntactically equal.
\qed
\end{proof}
The fact that derivability of equations from the equations of the
initial algebra specification of $\Ratda$ is decidable is proved
similarly.

$\Ratiaz$, the inversive arithmetical meadow of rational numbers with
zero, is defined as follows:
\begin{ldispl}
\Ratiaz =
I(\sigiamdz,\eqnsiamdz \union
            \set{(1 + x^2 + y^2) \mmul (1 + x^2 + y^2)\minv = 1})\;.
\end{ldispl}
$\Ratdaz$, the divisive arithmetical meadow of rational numbers with
zero, is defined as follows:
\begin{ldispl}
\Ratdaz =
I(\sigdamdz,\eqnsdamdz \union
            \set{(1 + x^2 + y^2) \mdiv (1 + x^2 + y^2) = 1})\;.
\end{ldispl}

$\Ratiaz$ is a subalgebra of a reduct of $\Ratzi$.
First we prove a fact that is useful in the proving this result.
\begin{lemma}
\label{lemma-mult-inv-exist-ratiaz}
It follows from
$\eqnsiamdz \union \set{(1 + x^2 + y^2) \mmul (1 + x^2 + y^2)\minv = 1}$
that $\ul{n}$ has a multiplicative inverse for each
$n \in \Nat \diff \set{0}$.
\end{lemma}
\begin{proof}
In the proof of Theorem~\ref{theorem-init-alg-spec-ratzi}, it is among
other things proved that it follows from
$\eqnsiamdz \union \set{x + (-x) = 0} \union
   \set{(1 + x^2 + y^2) \mmul (1 + x^2 + y^2)\minv = 1}$
that $\ul{n}$ has a multiplicative inverse for each
$n \in \Nat \diff \set{0}$.
The proof concerned goes through because no use is made of the equation
$x + (-x) = 0$.
\qed
\end{proof}
\begin{theorem}
\label{theorem-init-alg-spec-ratiaz}
$\Ratiaz$ is the subalgebra of the $\sigiamdz$-reduct of $\Ratzi$ whose
domain is the set of all non-negative rational numbers.
\end{theorem}
\begin{proof}
Like in the case of Theorem~\ref{theorem-init-alg-spec-ratia}, it is
sufficient to prove that, for each closed term $t$ over the signature
$\sigiamd$, there exists a unique term $t'$ in the set
\begin{ldispl}
\set{\ul{0}} \union
\set{\ul{n} \mmul \ul{m}\minv \where
     n,m \in \Nat \diff \set{0} \And \nm{gcd}(n,m) = 1}
\end{ldispl}
such that
$\eqnsiamdz \union
 \set{(1 + x^2 + y^2) \mmul (1 + x^2 + y^2)\minv = 1} \deriv t = t'$.
Like in the case of Theorem~\ref{theorem-init-alg-spec-ratia}, this is
proved by induction on the structure of $t$, now using
Lemmas~\ref{lemma-numerals}, \ref{lemma-zero-ratiaz},
and~\ref{lemma-mult-inv-ratiaz}.
The proof is similar, but more complicated owing to:
(i)~the presence of terms of the form $0$;
(ii)~the presence of terms of the form $\ul{0}$ among the terms that
exist by the induction hypothesis;
(iii)~the absence of the axiom $x \mmul x\minv = 1$.
Because of the last point, use is made of
Lemma~\ref{lemma-mult-inv-exist-ratiaz}.
\qed
\end{proof}
The fact that $\Ratdaz$ is a subalgebra of a reduct of $\Ratzd$ is
proved similarly.

An alternative initial algebra specification of $\Ratiaz$ is obtained if
the equation $(1 + x^2 + y^2) \mmul (1 + x^2 + y^2)\minv = 1$ is
replaced by
$(x \mmul (x + y)) \mmul (x \mmul (x + y))\minv = x \mmul x\minv$.
\begin{theorem}
\label{theorem-alt-init-alg-spec-ratiaz}
$\Ratiaz \cong
 I(\sigiamdz,
   \eqnsiamdz \union
   \set{(x \mmul (x + y)) \mmul (x \mmul (x + y))\minv =
        x \mmul x\minv})$.
\end{theorem}
\begin{proof}
It is sufficient to prove that
$(x \mmul (x + y)) \mmul (x \mmul (x + y))\minv = x \mmul x\minv$
is valid in $\Ratiaz$ and
$(1 + x^2 + y^2) \mmul (1 + x^2 + y^2)\minv = 1$ is valid in
$I(\sigiamdz,
   \eqnsiamdz \union
   \set{(x \mmul (x + y)) \mmul (x \mmul (x + y))\minv =
        x \mmul x\minv})$.
It follows from Lemma~\ref{lemma-mult-inv-ratiaz}, and the
associativity and commutativity of ${}\mmul{}$, that
$(x \mmul (x + y)) \mmul (x \mmul (x + y))\minv = x \mmul x\minv \Iff
 (x \mmul x\minv) \mmul ((x + y) \mmul (x + y)\minv) =
 x \mmul x\minv$ is derivable from $\eqnsiamdz$.
This implies that
$(x \mmul (x + y)) \mmul (x \mmul (x + y))\minv = x \mmul x\minv$
is valid in $\Ratiaz$ iff
$(x \mmul x\minv) \mmul ((x + y) \mmul (x + y)\minv) = x \mmul x\minv$
is valid in $\Ratiaz$.
The latter is easily established by distinction
between the cases $x = 0$ and $x \neq 0$.
To show that
$(1 + x^2 + y^2) \mmul (1 + x^2 + y^2)\minv = 1$ is valid in
$I(\sigiamdz,
   \eqnsiamdz \union
   \set{(x \mmul (x + y)) \mmul (x \mmul (x + y))\minv =
        x \mmul x\minv})$,
it is sufficient to derive
$(1 + x^2 + y^2) \mmul (1 + x^2 + y^2)\minv = 1$ from
$\eqnsiamdz \union
 \set{(x \mmul x\minv) \mmul ((x + y) \mmul (x + y)\minv) =
      x \mmul x\minv}$.
The derivation is fully trivial with the exception of the first step,
viz.\ substituting $1$ for $x$ and $x^2 + y^2$ for $y$ in
$(x \mmul x\minv) \mmul ((x + y) \mmul (x + y)\minv) = x \mmul x\minv$.
\qed
\end{proof}
An alternative initial algebra specification of $\Ratdaz$ is obtained in
the same vein.

In $\Ratiaz$, the \emph{general inverse law}
$x \neq 0 \Implies x \mmul x\minv = 1$ is valid.
Derivability of equations from the equations of the alternative initial
algebra specification of $\Ratiaz$ and the general inverse law is
decidable.
First we prove a fact that is useful in proving this decidability
result.
\begin{lemma}
\label{lemma-derivability-gil}
For all $\sigiamd$-terms $t$ in which no other variables than
$x_1,\ldots,x_n$ occur, % we have
$\eqnsiamdz \union
 \set{(x \mmul (x + y)) \mmul (x \mmul (x + y))\minv =
      x \mmul x\minv} \union
 \set{x_1 \mmul x_1\minv = 1,\ldots,\linebreak[2]x_n \mmul x_n\minv = 1}
  \deriv_{x_1,\ldots,x_n} t \mmul t\minv = 1$.
\end{lemma}
\begin{proof}
The proof is easy by induction on the structure of $t$, using
Lemma~\ref{lemma-mult-inv-ratiaz}.
\qed
\end{proof}

\begin{theorem}
\label{theorem-decidability-ratiaz}
% !! \linebreak
For all $\sigiamdz$-terms $t$ and $t'$, it is decidable whether
$\eqnsiamdz \union \linebreak
 \set{(x \mmul (x + y)) \mmul (x \mmul (x + y))\minv =
      x \mmul x\minv} \union
 \set{x \neq 0 \Implies x \mmul x\minv = 1} \deriv t = t'$.
\end{theorem}
\begin{proof}
% !! \linebreak
Let $\eqnsiamdzx =
 \eqnsiamdz \union
 \set{(x \mmul (x + y)) \mmul (x \mmul (x + y))\minv =
      x \mmul x\minv} \union \linebreak
 \set{x \neq 0 \Implies x \mmul x\minv = 1}$.
We prove that $\eqnsiamdzx \deriv t = t'$ is decidable by induction on
the number of variables occurring in $t = t'$.
In the case where the number of variables is $0$, we have that
$\eqnsiamdzx \deriv t = t'$ iff $\Ratiaz \models t = t'$ iff
$\eqnsiamdz \union
 \set{(1 + x^2 + y^2) \mmul (1 + x^2 + y^2)\minv = 1} \deriv t = t'$.
The last is decidable because, by the proof of
Theorem~\ref{theorem-init-alg-spec-ratiaz}, there exist
unique terms $s$ and $s'$ in the set
$\set{\ul{0}} \union
 \set{\ul{n} \mmul \ul{m}\minv \where
      n,m \in \Nat \diff \set{0} \And \nm{gcd}(n,m) = 1}$
such that
$\eqnsiamdz \union
 \set{(1 + x^2 + y^2) \mmul (1 + x^2 + y^2)\minv = 1} \deriv t = s$ and
$\eqnsiamdz \union
 \set{(1 + x^2 + y^2) \mmul \linebreak[2] (1 + x^2 + y^2)\minv = 1}
  \deriv t' = s'$,
and inspection of that proof yields that there is an effective way to
find $s$ and $s'$.
Hence, in the case where the number of variables is $0$,
$\eqnsiamdzx \deriv t = t'$ is decidable.
In the case where the number of variables is $n + 1$, suppose that the
variables are $x_1,\ldots,x_{n+1}$.
Let $s$ be such that $\eqnsiamdz \deriv t = s$ and $s$ is either a
$\sigiamd$-term or the constant $0$ and let $s'$ be such that
$\eqnsiamdz \deriv t' = s'$ and $s'$ is either a $\sigiamd$-term or the
constant $0$.
Such $s$ and $s'$ exist by Lemma~\ref{lemma-zero-elim-ratiaz}, and
inspection of the proof of that lemma yields that there is an effective
way to find $s$ and $s'$.
We have that $\eqnsiamdzx \deriv t = t'$ iff
$\eqnsiamdzx \deriv s = s'$.
In the case where not both $s$ and $s'$ are $\sigiamd$-terms,
$\eqnsiamdzx \deriv s = s'$ only if $s$ and $s'$ are syntactically
equal.
Hence, in this case, $\eqnsiamdzx \deriv t = t'$ is decidable.
In the case where both $s$ and $s'$ are $\sigiamd$-terms, by the general
inverse law, we have that
$\eqnsiamdzx \deriv s = s'$ iff
$\eqnsiamdzx \deriv s\subst{0}{x_i} =
 s'\subst{0}{x_i}$
for all $i \in [1,n + 1]$ and
$\eqnsiamdzx \union
 \set{x_1 \mmul x_1\minv = 1,\ldots,x_{n+1} \mmul x_{n+1}\minv = 1}
  \deriv_{x_1,\ldots,x_{n+1}} s = s'$.
By Lemma~\ref{lemma-derivability-gil}, we have that
$\eqnsiamdzx \union
 \set{x_1 \mmul x_1\minv = 1,\ldots,x_{n+1} \mmul x_{n+1}\minv = 1}
  \deriv_{x_1,\ldots,x_{n+1}} s = s'$ iff
$\eqnsiamd \deriv s = s'$.
For each $i \in [1,n + 1]$,
$\eqnsiamdzx \deriv s\subst{0}{x_i} =
 s'\subst{0}{x_i}$
is decidable because the number of variables occurring in
$s\subst{0}{x_i} = s'\subst{0}{x_i}$ is $n$.
Moreover, we know from Theorem~\ref{theorem-decidability-ratia} that
$\eqnsiamd \deriv s = s'$ is decidable.
Hence, in the case where both $s$ and $s'$ are $\sigiamd$-terms,
$\eqnsiamdzx \deriv t = t'$ is decidable as well.
\qed
\end{proof}
The fact that derivability of equations from the equations of the
alternative initial algebra specification of $\Ratdaz$ and
$x \neq 0 \Implies x \mdiv x = 1$ is decidable is proved similarly.
It is an open problem whether derivability of equations from the
equations of the alternative initial algebra specifications of $\Ratiaz$
and $\Ratdaz$ is decidable.

The following are some outstanding questions with regard to arithmetical
meadows:
\begin{enumerate}
\item
Is the initial algebra specification of $\Ratzi$ a conservative
extension of the initial algebra specifications of $\Ratia$ and
$\Ratiaz$?
\item
Do $\Ratia$ and $\Ratiaz$ have initial algebra specifications that
constitute complete term rewriting systems (modulo associativity and
commutativity of ${}+{}$ and~${}\mmul{}$)?
\item
Do $\Ratia$ and $\Ratiaz$ have $\omega$-complete initial algebra
specifications?
\item
What are the complexities of derivability of equations from
$\eqnsiamd$ and
$\eqnsiamdz \union
 \set{(x \mmul (x + y)) \mmul (x \mmul (x + y))\minv =
      x \mmul x\minv,\;
      x \neq 0 \Implies x \mmul x\minv = 1}$?
\item
Is derivability of equations from
$\eqnsiamdz \union
 \set{(x \mmul (x + y)) \mmul (x \mmul (x + y))\minv =
      x \mmul x\minv} \deriv t = t'$ decidable?
\item
Do we have
$\Ratiaz \cong
 I(\sigiamdz,
   \eqnsiamdz \union \set{(1 + x^2) \mmul (1 + x^2)\minv = 1})$?
\end{enumerate}
These questions are formulated for the inversive case, but they have
counterparts for the divisive case of which some might lead to different
answers.

\section{Partial Meadows}
\label{sect-partial-meadows}

In this section, we introduce simple constructions of partial inversive
and divisive meadows from total ones.
Divisive meadows are more basic than inversive meadows if the partial
ones are considered as well.

We take the position that partial algebras should be made from total
ones.
For the particular case of meadows, this implies that relevant partial
meadows are obtained by making operations undefined for certain
arguments.

Let $\cM_i$ be an inversive meadow.
Then it makes sense to construct one partial inversive meadow from
$\cM_i$:
\begin{itemize}
\item
$0\minv \punch \cM_i$ is the partial algebra that is obtained from
$\cM_i$ by making $0\minv$ undefined.
\end{itemize}
Let $\cM_d$ be a divisive meadow.
Then it makes sense to construct two partial divisive meadows from
$\cM_d$:
\begin{itemize}
\item
$\Quant \mdiv 0 \punch \cM_d$ is the partial algebra that is obtained
from $\cM_d$ by making\linebreak[2] $q \mdiv 0$ undefined for all $q$ in
the domain of $\cM_d$;
\item
$(\Quant \diff \set{0}) \mdiv 0 \punch \cM_d$ is the partial algebra
that is obtained from $\cM_d$ by making $q \mdiv 0$ undefined for all
$q$ in the domain of $\cM_d$ different from $0$.
\end{itemize}

Clearly, the partial meadow constructions are special cases of a more
general partial algebra construction for which we have coined the term
\emph{punching}.
Presenting the details of the general construction is outside the scope
of the current paper.

Let $\cM_i$ be an inversive meadow and let $\cM_d$ be an divisive
meadow.
It happens that the projection $\imndmn$ recovers $0\minv \punch \cM_i$
from $\Quant \mdiv 0 \punch \cM_d$ as well as
$(\Quant \diff \set{0}) \mdiv 0 \punch \cM_d$, the projection $\dmnimn$
recovers $\Quant \mdiv 0 \punch \cM_d$ from $0\minv \punch \cM_i$, and
the projection $\dmnimn$ does not recover
$(\Quant \diff \set{0}) \mdiv 0 \punch \cM_d$ from
$0\minv \punch \cM_i$:
\begin{itemize}
\item
$0\minv$ is undefined in $0\minv \punch \cM_i$,
$\imndmn(0\minv) = 1 \mdiv 0$, and $1 \mdiv 0$ is undefined in
$\Quant \mdiv 0 \punch \cM_d$ and
$(\Quant \diff \set{0}) \mdiv 0 \punch \cM_d$;
\item
$x \mdiv 0$ is undefined in $\Quant \mdiv 0 \punch \cM_d$,
$\dmnimn(x \mdiv 0) = x \mmul (0\minv)$, and $x \mmul (0\minv)$ is
undefined in $0\minv \punch \cM_i$;
\item
$0 \mdiv 0 = 0$ in $(\Quant \diff \set{0}) \mdiv 0 \punch \cM_d$,
$\dmnimn(0 \mdiv 0) = 0 \mmul (0\minv)$, but $0 \mmul (0\minv)$ is
undefined in $0\minv \punch \cM_i$.
\end{itemize}
This uncovers that $(\Quant \diff \set{0}) \mdiv 0 \punch \cM_d$
expresses a view on the partiality of division by zero that cannot be
expressed if only multiplicative inverse is available.
Therefore, we take divisive meadows for more basic than inversive
meadows if their partial variants are considered as well.
Otherwise, we might take inversive meadows for more basic instead, e.g.\
because of supposed notational simplicity
(see Section~\ref{sect-inv-div-meadows}).
Thus, the move from a total algebra to a partial algebra may imply a
reversal of the preferred direction of projection from $\dmnimn$ to
$\imndmn$.
This shows that projection semantics is a tool within a setting: if the
setting changes, the tool, or rather its way of application, changes as
well.

Returning to $(\Quant \diff \set{0}) \mdiv 0 \punch \cM_d$, the question
remains whether the equation $0 \mdiv 0 = 0$ is natural.
The total cost $C_n$ of producing $n$ items of some product is often
viewed as the sum of a fixed cost $\nm{FC}$ and a variable cost
$\nm{VC}_n$.
Moreover, for $n \geq 1$, the variable cost $\nm{VC}_n$ of producing
$n$ items is usually viewed as $n$ times the marginal cost per item,
taking $\nm{VC}_n \mdiv n$ as the marginal cost per item.
For $n = 0$, the variable cost of producing $n$ items and the marginal
cost per item are both $0$.
This makes the equation $\nm{VC}_0 \mdiv 0 = 0$ natural.

The partial meadows that we are most interested in are the three partial
meadows of rational numbers that can be obtained from $\Ratzi$ and
$\Ratzd$ by means of the partial meadow constructions introduced above:
\begin{ldispl}
\begin{eqncol}
0\minv \punch \Ratzi\;,
\qquad
\Quant \mdiv 0 \punch \Ratzd\;,
\qquad
(\Quant \diff \set{0}) \mdiv 0 \punch \Ratzd\;.
\end{eqncol}
\end{ldispl}

Notice that these partial algebras have been obtained by means of the
well-known initial algebra construction and a straightforward partial
algebra construction.
This implies that only equational logic for total algebras has been used
as a logical tool for their construction, like in case of $\Ratzi$ and
$\Ratzd$.
The approach followed here contrasts with the usual approach where a
special logic for partial algebras would be used for the construction of
partial algebras (see e.g.~\cite{CMR99a}).

We believe that many complications and unclarities in the development of
the theories of the partial algebras are avoided by not using some logic
of partial functions as a logical tool for their construction.
Having constructed $0\minv \punch \Ratzi$ in the way described above,
the question whether it satisfies the equation $0\minv = 0\minv$ and
related questions are still open because the logic of partial functions
to be used when working with $0\minv \punch \Ratzi$ has not been fixed
yet.
This means that it is still a matter of design which logic of partial
functions will be used when working with this partial algebra.%
\footnote
{A relevant survey and discussion of logics of partial functions can be
 found in Sections~7--9 of~\cite{BM09g}.
 The rest of that paper is fully included in the current paper.
}
As soon as the logic is fixed, the above-mentioned questions are no
longer open: it is anchored in the logic whether $0\minv = 0\minv$ is
satisfied, $0\minv \neq 0\minv$ is satisfied, or neither of the two is
satisfied.
Similar remarks apply to the other two partial algebras introduced
above.

Many people prefer $0\minv \punch \Ratzi$ to any other inversive algebra
of rational numbers.
It is likely that this is because $x \mmul x\minv = 1$ serves as an
implicit definition of ${}\minv$ in  $0\minv \punch \Ratzi$.

From the partial meadows of rational numbers introduced above,
$0\minv \punch \Ratzi$ and $\Quant \mdiv 0 \punch \Ratzd$ correspond
most closely to the prevailing viewpoint on the status of $1 \mdiv 0$ in
theoretical computer science that is mentioned in
Section~\ref{sect-viewpoints-div-by-zero}.
In the sequel, we will focus on $\Quant \mdiv 0 \punch \Ratzd$ because
the divisive notation is used more often than the inversive notation.

\section{Partial Arithmetical Meadows with Zero}
\label{sect-partial-arith-meadows}

In this section, we introduce simple constructions of partial inversive
and divisive arithmetical meadows with zero from total ones.
The constructions in question are variants of the constructions of
partial inversive and divisive meadows introduced in
Section~\ref{sect-partial-meadows}.

Let $\Mdiaz$ be an inversive arithmetical meadow with zero.
Then it makes sense to construct one partial inversive arithmetical
meadow with zero from $\Mdiaz$:
\begin{itemize}
\item
$0\minv \punch \Mdiaz$ is the partial algebra that is obtained from
$\Mdiaz$ by making $0\minv$ undefined.
\end{itemize}
Let $\Mddaz$ be a divisive arithmetical meadow with zero.
Then it makes sense to construct two partial divisive arithmetical
meadows with zero from $\Mddaz$:
\begin{itemize}
\item
$\Quant \mdiv 0 \punch \Mddaz$ is the partial algebra that is obtained
from $\Mddaz$ by making $q \mdiv 0$ undefined for all $q$ in the domain
of $\Mddaz$;
\item
$(\Quant \diff \set{0}) \mdiv 0 \punch \Mddaz$ is the partial algebra
that is obtained from $\Mddaz$ by making $q \mdiv 0$ undefined for all
$q$ in the domain of $\Mddaz$ different from $0$.
\end{itemize}

The following partial arithmetical meadows of rational numbers with zero
can be obtained from $\Ratzi$ and $\Ratzd$ by means of the partial
meadow constructions introduced above:
\pagebreak[2]
\begin{ldispl}
\begin{eqncol}
0\minv \punch \Ratiaz\;,
\qquad
\Quant \mdiv 0 \punch \Ratdaz\;,
\qquad
(\Quant \diff \set{0}) \mdiv 0 \punch \Ratdaz\;.
\end{eqncol}
\end{ldispl}

At first sight, the absence of the additive inverse operator does not
seem to add anything new to the treatment of partial meadows
in Section~\ref{sect-partial-meadows}.
However, this is not quite the case.
Consider $0\minv \punch \Ratiaz$.
In the case of this algebra, there is a useful syntactic criterion for
``being defined''.
The set $\Def$ of defined terms and the auxiliary set $\Nz$ of non-zero
terms can be inductively defined by:
\begin{itemize}
\item
$1 \in \Nz$;
\item
if $x \in \Nz$, then $x + y \in \Nz$ and $y + x \in \Nz$;
\item
if $x \in \Nz$ and $y \in \Nz$, then $x \mmul y \in \Nz$;
\item
if $x \in \Nz$, then $x\minv \in \Nz$;
\item
$0 \in \Def$;
\item
if $x \in \Nz$, then $x \in \Def$;
\item
if $x \in \Def$ and $y \in \Def$, then $x + y \in \Def$ and
$x \mmul y \in \Def$.
\end{itemize}
This indicates that the absence of the additive inverse operator allows
a typing based solution to problems related to ``division by zero'' in
elementary school mathematics.
So there may be a point in dealing first and thoroughly with
non-negative rational numbers in a setting where division by zero is not
defined.

Working in $\Ratia$ simplifies matters even more because there is no
distinction between terms and defined terms.
Again, this may be of use in the teaching of mathematics at elementary
school.

\section{Imperative Meadows of Rational Numbers}
\label{sect-imperative-meadows}

In this section, we introduce imperative inversive and divisive meadows
of rational numbers.

An imperative meadow of rational numbers is a meadow of rational numbers
together with an imperative to comply with a very strong convention with
regard to the use of the multiplicative inverse or division operator.

Like with the partial meadows of rational numbers, we introduce three
imperative meadows of rational numbers:
\begin{itemize}
\item
$0\minv \imper \Ratzi$ is $\Ratzi$ together with the imperative to
comply with the convention that $q\minv$ is not used with $q = 0$;
\item
$\Quant \mdiv 0 \imper \Ratzd$ is $\Ratzd$ together with the
imperative to comply with the convention that $p \mdiv q$ is not used
with $q = 0$;
\item
$(\Quant \diff \set{0}) \mdiv 0 \imper \Ratzd$ is $\Ratzd$ together
with the imperative to comply with the convention that $p \mdiv q$ is
not used with $q = 0$ if $p \neq 0$.
\end{itemize}
The conventions are called the \emph{relevant inversive convention},
the \emph{relevant division convention} and
the \emph{liberal relevant division convention}, respectively.

The conventions are very strong in the settings in which they must be
complied with.
For example, the relevant division convention is not complied with if
the question ``what is $1 \mdiv 0$'' is posed.
Using $1 \mdiv 0$ is disallowed, although we know that $1 \mdiv 0 = 0$
in~$\Ratzd$.

The first two of the imperative meadows of rational numbers introduced
above correspond most closely to the second of the two prevailing
viewpoints on the status of $1 \mdiv 0$ in mathematics that are
mentioned in Section~\ref{sect-viewpoints-div-by-zero}.
In the sequel, we will focus on $\Quant \mdiv 0 \imper \Ratzd$ because
the divisive notation is used more often than the inversive notation.

\section{Discussion on the Relevant Division Convention}
\label{sect-rel-div-conv}

In this section, we discuss the relevant division convention, i.e.\ the
convention that plays a prominent part in imperative meadows.

The existence of the relevant division convention can be explained by
assuming a context in which two phases are distinguished: a definition
phase and a working phase.
A mathematician experiences these phases in this order.
In the definition phase, the status of $1 \mdiv 0$ is dealt with
thoroughly so as to do away with the necessity of reflection upon it
later on.
As a result, $\Ratzd$ and the relevant division convention come up.
In the working phase, $\Ratzd$ is simply used in compliance with the
relevant division convention when producing mathematical texts.
Questions relating to $1 \mdiv 0$ are understood as being part of the
definition phase, and thus taken out of mathematical practice.
This corresponds to a large extent with how mathematicians work.

In the two phase context outlined above, the definition phase can be
made formal and logical whereas the results of this can be kept out of
the working phase.
Indeed, in mathematical practice, we find a world where logic does not
apply and where validity of work is not determined by the intricate
details of a very specific formal definition but rather by the consensus
obtained by a group of readers and writers.

Whether a mathematical text, including definitions, questions, answers,
conjectures and proofs, complies with the relevant division convention
is a judgement that depends on the mathematical knowledge of the reader
and writer.
For example, $\Forall{x}{(x^2 + 1) \mdiv (x^2 + 1) = 1}$ complies with
the relevant division convention because the reader and writer of it
both know that $\Forall{x}{x^2 + 1 \neq 0}$.

Whether a mathematical text complies with the relevant division
convention may be judged differently even with sufficient mathematical
knowledge.
This is illustrated by the following mathematical text, where $>$ is the
usual ordering on the set of rational numbers:
\begin{quote}
\textbf{Theorem.}\,\,
If $p \mdiv q = 7$ then
$\displaystyle \frac{q^2 + p \mdiv q - 7}{q^4 + 1} > 0$.

\textit{Proof.}\,\,
Because $q^4 + 1 > 0$, it is sufficient to show that
$q^2 + p \mdiv q - 7 > 0$.
It follows from $p \mdiv q = 7$ that $q^2 + p \mdiv q - 7 = q^2$,
and $q^2 > 0$ because $q \neq 0$ (as $p \mdiv q = 7$). \qed
\end{quote}
Reading from left to right, it cannot be that first $p \mdiv q$ is used
while knowing that $q \neq 0$ and that later on $q \neq 0$ is inferred
from the earlier use of $p \mdiv q$.
However, it might be said that the first occurrence of the text fragment
$p \mdiv q = 7$ introduces the knowledge
that $q \neq 0$ at the right time, i.e.\ only after it has been entirely
read.

The possibility of different judgements with sufficient mathematical
knowledge looks to be attributable to the lack of a structure theory of
mathematical text.
However, with a formal structure theory of mathematical text, we still
have to deal with the fact that compliance with the relevant division
convention is undecidable.

The imperative to comply with the relevant division conventions boils
down to the disallowance of the use of $1 \mdiv 0$,
$1 \mdiv (1 + (-1))$, etcetera in mathematical text.
The usual explanation for this is the non-existence of a $z$ such that
$0 \mmul z = 1$.
This makes the legality of $1 \mdiv 0$ comparable to the legality of
$\sum_{m = 1}^\infty 1 \mdiv m$, because of the non-existence of the
limit of $\tup{\sum_{m = 1}^{n+1} 1 \mdiv m}_{n \in \Nat}$.
However, a mathematical text may contain the statement
``$\sum_{m = 1}^\infty 1 \mdiv m$ is divergent''.
That is, the use of $\sum_{m = 1}^\infty 1 \mdiv m$ is not disallowed.
So the fact that there is no rational number that mathematicians intend
to denote by an expression does not always lead to the disallowance of
its use.

In the case of $1 \mdiv 0$, there is no rational number that
mathematicians intend to denote by $1 \mdiv 0$, there is no real number
that mathematicians intend to denote by $1 \mdiv 0$, there is no complex
number that mathematicians intend to denote by $1 \mdiv 0$, etcetera.
A slightly different situation arises with $\sqrt{2}$: there is no
rational number that mathematicians intend to denote by $\sqrt{2}$, but
there is a real number that mathematicians intend to denote by
$\sqrt{2}$.
It is plausible that the relevant division convention has emerged
because there is no well-known extension of the field of rational
numbers with a number that mathematicians intend to denote by
$1 \mdiv 0$.

\section{Partial Meadows and Logics of Partial Functions}
\label{sect-inadequacy-LPF}

In this section, we adduce arguments in support of the statement that
partial meadows together with logics of partial functions do not quite
explain how mathematicians deal with $1 \mdiv 0$ in mathematical works.
It needs no explaining that a real proof of this statement is out of the
question.
However, we do not preclude the possibility that more solid arguments
exist.
Moreover, as it stands, it is possible that our argumentation leaves
room for controversy.

In the setting of a logic of partial functions, there may be terms whose
value is undefined.
Such terms are called non-denoting terms.
Moreover, often three truth values, corresponding to true, false and
neither-true-nor-false, are considered.
These truth values are denoted by $\True$, $\False$, and $\Undef$,
respectively.

In logics of partial functions, three different kinds of equality are
found (see e.g.~\cite{MR91a}).
They only differ in their treatment of non-denoting terms:
\begin{itemize}
\item
\emph{weak equality}: if either $t$ or $t'$ is non-denoting, then the
truth value of $t = t'$ is $\Undef$;
\item
\emph{strong equality}: if either $t$ or $t'$ is non-denoting, then the
truth value of $t = t'$ is $\True$ whenever both $t$ and $t'$ are
non-denoting and $\False$ otherwise;
\item
\emph{existential equality}: if either $t$ or $t'$ is non-denoting, then
the truth value of $t = t'$ is $\False$.
\end{itemize}

With strong equality, the truth value of $1 \mdiv 0 = 1 \mdiv 0 + 1$ is
$\True$.
This does not at all fit in with mathematical practice.
With existential equality, the truth value of $1 \mdiv 0 = 1 \mdiv 0$ is
$\False$.
This does not at all fit in with mathematical practice as well.
Weak equality is close to mathematical practice: the truth value of an
equation is neither $\True$ nor $\False$ if a term of the form
$p \mdiv q$ with $q = 0$ occurs in it.

This means that the classical logical connectives and quantifiers must
be extended to the three-valued case.
Many ways of extending them must be considered uninteresting for a logic
of partial functions because they lack an interpretation of the third
truth value that fits in with its origin: dealing with non-denoting
terms.
If those ways are excluded, only four ways to extend the classical
logical connectives to the three-valued case remain
(see e.g.~\cite{BBR95a}).
Three of them are well-known: they lead to Bochvar's strict
connectives~\cite{Boc39a}, McCarthy's sequential
connectives~\cite{McC63a}, and Kleene's monotonic
connectives~\cite{Kle38a}.
The fourth way leads to McCarthy's sequential connectives with the role
of the operands of the binary connectives reversed.

In mathematical practice, the truth value of
$\Forall{x}{x \neq 0 \Implies x \mdiv x = 1}$ is considered $\True$.
Therefore, the truth value of $0 \neq 0 \Implies 0 \mdiv 0 = 1$ is
$\True$ as well.
With Bochvar's connectives, the truth value of this formula is $\Undef$.
With McCarthy's or Kleene's connectives the truth value of this formula
is $\True$.
However, unlike with Kleene's connectives, the truth value of the
seemingly equivalent $0 \mdiv 0 = 1 \Or 0 = 0$ is $\Undef$ with
McCarthy's connectives.
Because this agrees with mathematical practice, McCarthy's connectives
are closest to mathematical practice.

The conjunction and disjunction connectives of Bochvar and the
conjunction and disjunction connectives of Kleene have natural
generalizations to quantifiers, which are called Bochvar's quantifiers
and Kleene's quantifiers, respectively.
Both Bochvar's quantifiers and Kleene's quantifiers can be considered
generalizations of the conjunction and disjunction connectives of
McCarthy.%
\footnote
{In~\cite{KTB91a}, Bochvar's quantifiers are called McCarthy's
 quantifiers, but McCarthy combines his connectives with Kleene's
 quantifiers (see e.g.~\cite{Kle38a}).}

With Kleene's quantifiers, the truth value of
$\Forall{x}{x \mdiv x = 1}$ is $\Undef$ and the truth value of
$\Exists{x}{x \mdiv x = 1}$ is $\True$.
The latter does not at all fit in with mathematical practice.
Bochvar's quantifiers are close to mathematical practice: the truth
value of a quantified formula is neither $\True$ nor $\False$ if it
contains a term of the form $p \mdiv q$ where $q$ has a closed
substitution instance $q'$ with $q' = 0$.

What precedes suggest that mathematical practice is best approximated by
a logic of partial functions with weak equality, McCarthy's connectives
and Bochvar's quantifiers.
We call this logic the \emph{logic of partial meadows}, abbreviated
$\LPMd$.

In order to explain how mathematicians deal with $1 \mdiv 0$ in
mathematical works, we still need the convention that a sentence is not
used if its truth value is neither $\True$ nor $\False$.
We call this convention the \emph{two-valued logic convention}.

$\LPMd$ together with the imperative to comply with the two-valued logic
convention gets us quite far in explaining how mathematicians deal with
$1 \mdiv 0$ in mathematical works.
However, in this setting, not only the truth value of
$0 \neq 0 \Implies 0 \mdiv 0 = 1$ is $\True$, but also the truth value
of $0 = 0 \Or 0 \mdiv 0 = 1$ is $\True$.\linebreak[2]
In our view, the latter does not fit in with how mathematicians deal
with $1 \mdiv 0$ in mathematical works.
Hence, we conclude that $\LPMd$, even together with the imperative to
comply with the two-valued logic convention, fails to provide a
convincing account of how mathematicians deal with $1 \mdiv 0$ in
mathematical works.

\section{Concluding Remarks}
\label{sect-conclusions}

We have made a formal distinction between inversive meadows and divisive
meadows.
We have given finite equational specifications of the class of all
inversive meadows, the class of all divisive meadows, and arithmetical
variants of them.
We have also given finite equational specifications whose initial
algebras are inversive meadows of rational numbers, divisive meadows of
rational numbers, and arithmetical variants of them.
We have introduced and discussed constructions of variants of inversive
meadows, divisive meadows, and arithmetical variants of them with a
partial multiplicative inverse or division operation from the total
ones.
Moreover, we have given an account on how mathematicians deal with
$1 \mdiv 0$ in mathematical work, using the concept of an imperative
meadow, and have made plausible that a convincing account of how
mathematicians deal with $1 \mdiv 0$ by means of some logic of partial
functions is not attainable.

We have obtained various algebras of rational numbers by means of the
well-known initial algebra construction and, in some cases, the
above-mentioned partial algebra constructions.
This implies that in all cases only equational logic for total algebras
has been used as a logical tool for their construction.
In this way, we have avoided choosing or developing an appropriate
logic, which we consider a design problem of logics, not of data types.
We claim that, viewed from the theory of abstract data types, the way in
which partial algebras are constructed in this paper is the preferred
way.
Its main advantage is that no decision need to be taken in the course of
the construction about matters concerning the logic to be used when
working with the partial algebras in question.
For that reason, we consider it useful to generalize the partial algebra
constructions on inversive and divisive meadows to a partial algebra
construction that can be applied to any total algebra.

Our account on how mathematicians deal with $1 \mdiv 0$ in mathematical
work makes use of the concept of an imperative meadow.
This concept is a special case of the more general concept of an
imperative algebra, i.e.\ an algebra together with the imperative to
comply with one or more conventions about its use.
An example of an imperative algebra is imperative stack algebra: stack
algebra, whose signature consists of $\nm{empty}$, $\nm{push}$,
$\nm{pop}$ and $\nm{top}$, together with the imperative to comply with
the convention that $\nm{top}(s)$ is not used with $s = \nm{empty}$.
In~\cite{BM09k}, this idea is successfully used in work on the
autosolvability requirement inherent in Turing's result regarding the
undecidability of the halting problem.

We have argued that a logic of partial functions with weak equality,
McCarthy's connectives and Bochvar's quantifiers, together with the
imperative to comply with the convention that sentences whose truth
value is neither $\True$ nor $\False$ are not used, approximates
mathematical practice best, but after all fails to provide a convincing
account of how mathematicians deal with $1 \mdiv 0$ in mathematical
works.
To our knowledge, there are no published elaborations on such a logic of
partial functions.
In most logics of partial functions that have been proposed by computer
scientists, including PPC~\cite{Hoo77a}, LPF~\cite{BCJ84a},
PFOL~\cite{GL90a} and WS~\cite{Owe93a}, weak equality, Kleene's
connectives and Kleene's quantifiers are taken as basic.

The axioms of an inversive meadow forces that the equation $0\minv = 0$
holds.
It happens that this equation is used for technical convenience in
several other places, see e.g.~\cite{Hod93a,Har98a}.
The axioms of a divisive meadow forces that the equation $x \mdiv 0 = 0$
holds.
One of the few published pieces of writing about this equation that we
have been able to trace is~\cite{McD76a}.

We have answered a number of questions about arithmetical meadows of
rational numbers, and stated a number of outstanding questions about
them.
We remark that the name arithmetical algebra is not always used in the
same way as Peacock~\cite{Pea1830a} used it.
It is sometimes difficult to establish whether the notion in question is
related to Peacock's notion of arithmetical algebra.
For example, it is not clear to us whether the notion of arithmetical
algebra defined in~\cite{Pix72a} is related to Peacock's notion of
arithmetical algebra.

The theory of meadows has among other things been applied
in~\cite{BPZ07a,BB09b}.

\appendix

\section{Modular Specification of Divisive Meadows}
\label{sect-module-algebra}

In this section, we give a modular specification of divisive meadows
using basic module algebra~\cite{BHK88a}.

\BMAfol\ (Basic Module Algebra for \emph{f}irst-\emph{o}rder
\emph{l}ogic specifications) is a many-sorted equational theory of
modules which covers the concepts on which the key modularization
mechanisms found in existing specification formalisms are based.
The signature of \BMAfol\ includes among other things:
\begin{itemize}
\item
the sorts $\AtSig$ of \emph{atomic signatures},
$\Ren$ of \emph{atomic renamings},
$\Sig$ of \emph{signatures}, and
$\Mod$ of \emph{modules};
\item
the binary \emph{deletion} operator
$\funct{\delete}{\AtSig \x \Sig}{\Sig}$;
\item
the unary \emph{signature} operator $\funct{\sig}{\Mod}{\Sig}$;
\item
for each first-order sentence $\phi$ over some signature,
the constant $\const{\sen{\phi}}{\Mod}$;
\item
the binary \emph{renaming application} operator
$\funct{\rename}{\Ren \x \Mod}{\Mod}$;
\item
the binary \emph{combination} operator
$\funct{\combin}{\Mod \x \Mod}{\Mod}$;
\item
the binary \emph{export} operator
$\funct{\export}{\Sig \x \Mod}{\Mod}$.
\end{itemize}
The axioms of \BMAfol\ as well as four different models for \BMAfol\ can
be found in~\cite{BHK88a}.
A useful derived operator is the \emph{hiding} operator
$\funct{\hiding}{\AtSig \x \Mod}{\Mod}$ defined by
$a \hiding X = (a \delete \sig(X)) \export X$.
Below, we will use the notational conventions introduced in Section~3.5
of~\cite{BHK88a}.

Let $\nm{Md_i}$ be the closed module expression corresponding to the
equations $\eqnsimd$, i.e.\
$\nm{Md_i} =
 \sen{(x + y) + z = x + (y + z)} \combin \cdots \combin
 \sen{x \mmul (x \mmul x\minv) = x}$.
We give a modular specification of divisive meadows using \BMAfol\ as
follows:
\begin{ldispl}
\nm{Md_d} =
\funcd{{\minv}}{\Quant \to \Quant} \hiding
(\nm{Md_i} \combin \sen{x \mdiv y = x \mmul (y\minv)})\;.
\end{ldispl}
In~\cite{BHK88a}, a semantic mapping $\nm{EqTh}$ is defined that gives,
for each closed module expression, its equational theory.
We have the following theorem:
\begin{theorem}
\label{theorem-module-algebra-1}
$\nm{EqTh}(\nm{Md_d})$ is the equational theory associated with the
equational specification of divisive meadows given in
Section~\ref{sect-inv-div-meadows}.
\end{theorem}
\begin{proof}
In~\cite{BHK88a}, a semantic mapping $\nm{Mod}$ is defined that gives,
for each closed module expression, its model class.
$\nm{Mod}$ and $\nm{EqTh}$ are defined such that $\nm{EqTh}(m)$ is the
equational theory of $\nm{Mod}(m)$ for each closed module expression
$m$.
Hence, it is sufficient to show that $\nm{Mod}(\nm{Md_d})$ is the class
of models of the equational specification of divisive meadows.
By the definition of $\nm{Mod}$, we have to show that:
(i)~the reduct to the signature of divisive meadows of each model of the
equational specification of inversive meadows extended with the equation
$x \mdiv y = x \mmul (y\minv)$ is a model of the equational
specification of divisive meadows;
(ii)~each model of the equational specification of divisive meadows can
be expanded with a multiplicative inverse operation satisfying
${(x\minv)}\minv = x$ and $x \mmul (x \mmul x\minv) = x$.
Using the equations from the equational specification of inversive
meadows and the equation $x \mdiv y = x \mmul (y\minv)$, it can easily
be proved by equational reasoning that all equations from the equational
specification of divisive meadows are satisfied by the reducts in
question.
Let ${}\minv$ be defined by $x\minv = 1 \mdiv x$.
Then, using the equations from the equational specification of divisive
meadows and the equation $x\minv = 1 \mdiv x$, it can easily
be proved by equational reasoning that the equations
${(x\minv)}\minv = x$ and $x \mmul (x \mmul x\minv) = x$ are satisfied
by the expansions in question.
\qed
\end{proof}

We give the following modular specification of reduced divisive meadows:
\begin{ldispl}
\begin{aeqns}
\nm{Md_{rd1}} & = &
\funcd{{\mmul}}{\Quant \x \Quant \to \Quant} \hiding \nm{Md_d}\;,
\\
\nm{Md_{rd2}} & = &
\funcd{-}{\Quant \to \Quant} \hiding
(\nm{Md_{rd1}} \combin \sen{x - y = x + (- y)})\;,
\\
\nm{Md_{rd3}} & = &
\funcd{{+}}{\Quant \x \Quant \to \Quant} \hiding \nm{Md_{rd2}}\;,
\\
\nm{Md_{rd}} & = &
\funcd{0}{\Quant} \hiding \nm{Md_{rd3}}\;.
\end{aeqns}
\end{ldispl}
We have the following theorem:
\begin{theorem}
\label{theorem-module-algebra-2}
$\nm{EqTh}(\nm{Md_{rd}})$ is the equational theory associated with the
equational specification of reduced divisive meadows given in
Section~\ref{sect-inv-div-meadows}.
\end{theorem}
\begin{proof}
The proof follows the same line as the proof of
Theorem~\ref{theorem-module-algebra-1}.
For the expansion, we define zero, addition, multiplication, and
additive inverse as follows: $0 = 1 - 1$, $x + y = x - ((1 - 1) - y)$,
$x \mmul y = x \mdiv (1 \mdiv y)$, and $- x = (1 - 1) - x$.
\qed
\end{proof}

\bibliographystyle{splncs03}
\bibliography{MD}

\begin{thebibliography}{10}
\providecommand{\url}[1]{\texttt{#1}}
\providecommand{\urlprefix}{URL }

\bibitem{BCJ84a}
Barringer, H., Cheng, J.H., Jones, C.B.: A logic covering undefinedness in
  program proofs. Acta Informatica  21(3),  251--269 (1984)

\bibitem{BB09b}
Bergstra, J.A., Bethke, I.: Straight-line instruction sequence completeness for
  total calculation on cancellation meadows. Theory of Computing Systems
  DOI: 10.1007/s00224-010-9272-9 (2010)

\bibitem{BBR95a}
Bergstra, J.A., Bethke, I., Rodenburg, P.H.: A propositional logic with 4
  values: True, false, divergent and meaningless. Journal of Applied
  Non-Classical Logic  5(2),  199--218 (1995)

\bibitem{BHK88a}
Bergstra, J.A., Heering, J., Klint, P.: Module algebra. Journal of the ACM
  37(2),  335--372 (1990)

\bibitem{BHT09a}
Bergstra, J.A., Hirshfeld, Y., Tucker, J.V.: Meadows and the equational
  specification of division. Theoretical Computer Science  410(12--13),
  1261--1271 (2009)

\bibitem{BHT09b}
Bergstra, J.A., Hirshfeld, Y., Tucker, J.V.: Skew meadows. {\tt
  arXiv:0901.0803v1 [math.RA]} at {\tt http://arxiv.org/} (January 2009)

\bibitem{BL02a}
Bergstra, J.A., Loots, M.E.: Program algebra for sequential code. Journal of
  Logic and Algebraic Programming  51(2),  125--156 (2002)

\bibitem{BM09h}
Bergstra, J.A., Middelburg, C.A.: Arithmetical meadows. 
  {\tt arXiv:0909.2088v1 [math.RA]} at {\tt http://arxiv.org/} (September
  2009)

\bibitem{BM09k}
Bergstra, J.A., Middelburg, C.A.: Instruction sequence processing operators.
  {\tt arXiv:0910.5564v3 [cs.LO]} at {\tt http://arxiv.org/} (October 2009)

\bibitem{BM09g}
Bergstra, J.A., Middelburg, C.A.: Inversive meadows and divisive meadows.
  {\tt arXiv:0907.0540v2 [math.RA]} at {\tt http://arxiv.org/} (July 2009) 

\bibitem{BM09j}
Bergstra, J.A., Middelburg, C.A.: Partial {Komori} fields and imperative
  {Komori} fields. 
  {\tt arXiv:0909.5271v1 [math.RA]} at {\tt http://arxiv.org/} (September
  2009)

\bibitem{BP08a}
Bergstra, J.A., Ponse, A.: A generic basis theorem for cancellation meadows.
  {\tt arXiv:0803.3969v2 [math.RA]} at {\tt http://arxiv.org/} (March 2008)

\bibitem{BPZ07a}
Bergstra, J.A., Ponse, A., van~der Zwaag, M.B.: Tuplix calculus. Scientific
  Annals of Computer Science  18,  35--61 (2008)

\bibitem{BT87a}
Bergstra, J.A., Tucker, J.V.: Algebraic specifications of computable and
  semicomputable data types. Theoretical Computer Science  50(2),  137--181
  (1987)

\bibitem{BT95a}
Bergstra, J.A., Tucker, J.V.: Equational specifications, complete term
  rewriting, and computable and semicomputable algebras. Journal of the ACM
  42(6),  1194--1230 (1995)

\bibitem{BT07a}
Bergstra, J.A., Tucker, J.V.: The rational numbers as an abstract data type.
  Journal of the ACM  54(2),  Article 7 (2007)

\bibitem{BR08a}
Bethke, I., Rodenburg, P.H.: The initial meadows. {\tt arXiv:0806.2256v1
  [math.RA]} at {\tt http://arxiv.org/} (June 2008)

\bibitem{BRS09a}
Bethke, I., Rodenburg, P.H., Sevenster, A.: The structure of finite meadows.
  {\tt arXiv:0903.1196v1 [cs.LO]} at {\tt http://arxiv.org/} (March 2009)

\bibitem{Boc39a}
Bochvar, D.A.: On a three-valued logical calculus and its application to the
  analysis of contradictions (in {Russian}). Mat{\'{e}}mati{\u{c}}eskij Sbornik
   4(46)(2),  287--308 (1938), a translation into English appeared in History
  and Philosophy of Logic \textbf{2}(1--2), 87--112 (1981)

\bibitem{CMR99a}
Cerioli, M., Mossakowski, T., Reichel, H.: From total equational to partial
  first order logic. In: Astesiano, E., Kreowski, H.J., Krieg-Br{\"{u}}ckner,
  B. (eds.) Algebraic Foundations of Systems Specification, pp. 31--104.
  Springer-Verlag, Berlin (1999)

\bibitem{GL90a}
{Gavilanes-Franco}, A., {Lucio-Carrasco}, F.: A first order logic for partial
  functions. Theoretical Computer Science  74(1),  37--69 (1990)

\bibitem{Goo79a}
Goodearl, K.R.: {Von Neumann} Regular Rings. Pitman, London (1979)

\bibitem{Har98a}
Harrison, J.: Theorem Proving with the Real Numbers. Distinguished
  Dissertations Series, Springer-Verlag, Berlin (1998)

\bibitem{Hod93a}
Hodges, W.A.: Model Theory, Encyclopedia of Mathematics and Its Applications,
  vol.~42. Cambridge University Press, Cambridge (1993)

\bibitem{Hoo77a}
Hoogewijs, A.: A calculus of partially defined predicates. Mathematical
  scripts, Rijksuniversiteit Gent (1977)

\bibitem{Kle38a}
Kleene, S.C.: On notation for ordinal numbers. Journal of Symbolic Logic  3(4),
   150--155 (1938)

\bibitem{Kle98a}
Kleiner, I.: A historically focused course in abstract algebra. Mathematics
  Magazine  71(2),  105--111 (1998)

\bibitem{Kom75a}
Komori, Y.: Free algebras over all fields and pseudo-fields. Report 10,
  pp.~9--15, Faculty of Science, Shizuoka University (1975)

\bibitem{KTB91a}
Konikowska, B., Tarlecki, A., Blikle, A.: A three-valued logic for software
  specification and validation. Fundamenta Informaticae  14,  411--453 (1991)

\bibitem{McC63a}
McCarthy, J.: A basis for a mathematical theory of computation. In: Braffort,
  P., Hirschberg, D. (eds.) Computer Programming and Formal Systems, pp.
  33--70. North-Holland, Amsterdam (1963)

\bibitem{McC64a}
McCoy, N.H.: The Theory of Rings. Macmillan, London (1964)

\bibitem{McD76a}
McDonnell, E.E.: Zero divided by zero. In: APL '76. pp. 295--296. ACM Press
  (1976)

\bibitem{MR91a}
Middelburg, C.A., {Renardel de Lavalette}, G.R.: {LPF} and {MPL}$_\omega$ -- a
  logical comparison of {VDM SL} and {COLD-K}. In: Prehn, S., Toetenel, W.J.
  (eds.) VDM~'91, Volume~1. Lecture Notes in Computer Science, vol. 551, pp.
  279--308. Springer-Verlag, Berlin (1991)

\bibitem{Ono83a}
Ono, H.: Equational theories and universal theories of fields. Journal of the
  Mathematical Society of Japan  35(2),  289--306 (1983)

\bibitem{Owe93a}
Owe, O.: Partial logics reconsidered: A conservative approach. Formal Aspects
  of Computing  5(3),  208--223 (1993)

\bibitem{Pea1830a}
Peacock, G.: A Treatise on Algebra. J. \& J.J. Deighton, Cambridge (1830)

\bibitem{Pix72a}
Pixley, A.F.: Completeness in arithmetical algebras. Algebra Universalis  2(1),
   179--196 (1972)

\bibitem{Ver78a}
{Verloren van Themaat}, W.A.: Right-divisive groups. Notre Dame Journal of
  Formal Logic  19(1),  137--140 (1978)

\bibitem{Vis06a}
Visser, A.: Categories of theories and interpretations. In: Enayat, A.,
  Kalantari, I., Moniri, M. (eds.) Logic in Tehran 2003. Lecture Notes in
  Logic, vol.~26, pp. 284–--341. Association for Symbolic Logic (2006)

\bibitem{Yam63a}
Yamada, M.: Inversive semigroups {I}. Proceedings of Japan Academy  39(2),
  100--103 (1963)

\end{thebibliography}

% \par \vfill \par \noindent DRAFT of \today

\end{document}